\documentclass[11pt]{article} 
\usepackage[utf8]{inputenc} 
\usepackage[margin=1in]{geometry} 
\usepackage{graphicx} 

\usepackage{booktabs} 
\usepackage{array} 
\usepackage{paralist} 
\usepackage{verbatim} 
\usepackage{mathrsfs}
\usepackage{amssymb}
\usepackage{amsthm}
\usepackage{amsmath,amsfonts,amssymb}
\usepackage{esint}
\usepackage{graphics}
\usepackage{enumerate}
\usepackage{mathtools}
\usepackage{xfrac}
\usepackage{nicefrac}
\usepackage{subcaption}
\usepackage{stmaryrd}
\usepackage[normalem]{ulem}
\usepackage{cancel}
\usepackage{enumitem}
\usepackage{fancyvrb}
\usepackage{bm}

\allowdisplaybreaks[4]

\usepackage[usenames,dvipsnames]{xcolor}
\usepackage[colorlinks=true, pdfstartview=FitV, linkcolor=blue, citecolor=blue, urlcolor=blue]{hyperref}
\usepackage[normalem]{ulem}

\usepackage{tikz}
\usetikzlibrary{calc}
\usepackage{pgf}
\usetikzlibrary{external}
\numberwithin{equation}{section}
\numberwithin{figure}{section}

\newtheorem{theorem}{Theorem}[section]

\newtheorem{corollary}[theorem]{Corollary}
\newtheorem{proposition}[theorem]{Proposition}
\newtheorem{lemma}[theorem]{Lemma}

\theoremstyle{definition}
\newtheorem{definition}[theorem]{Definition}

\newtheorem{remark}[theorem]{Remark}

\usepackage{dsfont}
\usepackage{bbm}

\newcommand{\N}{\mathbb{N}}
\newcommand{\R}{\mathbb{R}}

\newcommand{\cA}{\mathcal{A}}

\newcommand{\cQ}{\mathcal{Q}}

\newcommand{\cL}{\mathcal{L}}

\newcommand\sfK{{\boldsymbol{\mathsf K}}}

\newcommand\sfk{{\boldsymbol{\mathsf k}}}
\newcommand\sfI{{\boldsymbol{\mathsf I}}}
\newcommand\sfV{{\boldsymbol{\mathsf V}}}

\newcommand\sfT{{\boldsymbol{\mathsf T}}}

\newcommand\sfG{{\boldsymbol{\mathsf G}}}
\newcommand\sfb{{\boldsymbol{\mathsf b}}}
\newcommand\sfY{{\boldsymbol{\mathsf Y}}}
\newcommand\sfS{{\boldsymbol{\mathsf S}}}
\newcommand\sfQ{{\boldsymbol{\mathsf Q}}}
\newcommand\sfZ{{\boldsymbol{\mathsf Z}}}
\newcommand\sfH{{\boldsymbol{\mathsf H}}}
\newcommand\sfW{{\boldsymbol{\mathsf W}}}
\newcommand\sfA{{\boldsymbol{\mathsf A}}}

\newcommand{\eps}{\varepsilon}

\newcommand{\1}{\mathbf{1}}

\newcommand{\data}{\textnormal{data}}
\renewcommand{\rho}{\varrho}

\DeclareMathOperator{\supp}{supp}

\newcommand{\esssup}{\operatornamewithlimits{ess\,sup}}
\newcommand{\essinf}{\operatornamewithlimits{ess\,inf}}

\DeclareMathOperator{\tail}{Tail}

\renewcommand{\d}{\mathrm{\,d}}
\def\dxy{\,{\mathrm d}x{\mathrm  d}y}
\def\dxt{\,{\mathrm  d}x{\mathrm d}t}

\def\dxk{\,{\mathrm  d}x{\mathrm  d}k}
\def\dtk{\,{\mathrm  d}t{\mathrm  d}k}

\def\dytau{\,{\mathrm  d}y{\mathrm  d}\tau}
\def\dtaux{\,{\mathrm  d}\tau{\mathrm  d}x}
\def\dxyt{\,{\mathrm d}x{\mathrm  d}y{\mathrm  d}t}

\def\dxtk{\,{\mathrm d}x{\mathrm  d}t{\mathrm  d}k}

\renewcommand{\leq}{\leqslant}
\renewcommand{\geq}{\geqslant}
\renewcommand{\subset}{\subseteq}
\renewcommand{\supset}{\supseteq}

\def\Xint#1{\mathchoice
{\XXint\displaystyle\textstyle{#1}}%
{\XXint\textstyle\scriptstyle{#1}}%
{\XXint\scriptstyle\scriptscriptstyle{#1}}%
{\XXint\scriptscriptstyle\scriptscriptstyle{#1}}%
\!\int}
\def\XXint#1#2#3{{\setbox0=\hbox{$#1{#2#3}{\int}$}
\vcenter{\hbox{$#2#3$}}\kern-.5\wd0}}
\def\dashint{\Xint-}

\def\Xiint#1{\mathchoice
    {\XXiint\displaystyle\textstyle{#1}}%
    {\XXiint\textstyle\scriptstyle{#1}}%
    {\XXiint\scriptstyle\scriptscriptstyle{#1}}%
    {\XXiint\scriptscriptstyle\scriptscriptstyle{#1}}%
    \!\iint}
\def\XXiint#1#2#3{\setbox0=\hbox{$#1{#2#3}{\iint}$}
    \vcenter{\hbox{$#2#3$}}\kern-0.5\wd0}
\def\biint{\Xiint{-\!-}}

\makeatletter
\newcommand\avsuminner[2]{%
  {\sbox0{$\m@th#1\sum$}%
   \vphantom{\usebox0}%
   \ooalign{%
     \hidewidth
     \smash{\,\rule[.23em]{8.8pt}{1.1pt} \relax}%
     \hidewidth\cr
     $\m@th#1\sum$\cr
   }%
  }%
}
\makeatother

\makeatletter
\newcommand\avsuminnerr[2]{%
  {\sbox0{$\m@th#1\sum$}%
   \vphantom{\usebox0}%
   \ooalign{%
     \hidewidth
     \smash{\,\rule[.23em]{6pt}{0.7pt} \relax}%
     \hidewidth\cr
     $\m@th#1\sum$\cr
   }%
  }%
}
\makeatother

\let\originalleft\left
\let\originalright\right
\renewcommand{\left}{\mathopen{}\mathclose\bgroup\originalleft}
\renewcommand{\right}{\aftergroup\egroup\originalright}

\usepackage{titlesec}

\newcommand{\addperiod}[1]{#1.}
\titleformat{\section}{\centering\normalfont\Large}{\thesection.}{0.5em}{}
\titleformat*{\subsection}{\bfseries}
\titleformat{\subsubsection}[runin]{\normalfont\bfseries}{\thesubsubsection.}{0.5em}{\addperiod}
\titleformat*{\paragraph}{\bfseries}
\titleformat*{\subparagraph}{\large\bfseries}

\title{Harnack estimates for the nonlocal Trudinger equation}

\author{Simone Ciani
\thanks{Department of Mathematics of the University of Bologna, Piazza Porta San Donato, 5, 40126 Bologna, Italy.
{\footnotesize \href{mailto:simone.ciani3@unibo.it}{simone.ciani3@unibo.it}. 
}
}
  \and
Kenta Nakamura
\thanks{Institute of Natural Sciences, Nihon University, Tokyo, Japan.
{\footnotesize \href{mailto:kentanak55@gmail.com}{kentanak55@gmail.com} (corresponding author). 
}
}
}

\usepackage[nottoc,notlot,notlof]{tocbibind}
\usepackage{fancyhdr}
\usepackage{extramarks}
 \date{\today}

\begin{document}

\maketitle

\begingroup
\renewcommand{\thefootnote}{\ifcase\value{footnote}\or*\or**\fi}

\endgroup

\begin{abstract}
We carry on a study of the point-wise regularity theory for the nonlocal Trudinger equation, with measurable and bounded singular kernel. We establish refined quantitative upper bounds, weak and strong parabolic Harnack inequalities, both under optimal tail conditions. Our analysis relies on the adaptation of a refined De Giorgi-Moser machinery based on the parabolic approach {\it à la Di Benedetto}, specifically designed to overcome the technical intricacies of the nonlocal, doubly nonlinear framework.
\end{abstract}

\setcounter{tocdepth}{2}
\tableofcontents

\section{Introduction}

\subsection{Motivation and overview} We address the regularity theory for nonlocal Trudinger's operators in bounded domains, namely 
\begin{equation}
\partial_t \Big(|u|^{p-2}u\Big)+\cL u=0 \quad \mbox{in} \quad \Omega_T:=\Omega \times (0,T) \subset \R^{d+1},
\label{e.NT}
\end{equation}
where the nonlocal integro-differential operator $\cL$ is formally defined as
\begin{align*}
\cL u(x,t)&:=\mathrm{p.v.}\int_{\R^d}2\left|u(x,t)-u(y,t)\right|^{p-2}\left(u(x,t)-u(y,t)\right)\sfk (x,y,t)\d{y},
\end{align*}
and the kernel $\sfk: \R^d \times \R^d \times (0,T) \to [0,\infty)$ is measurable and satisfies the condition  
\begin{equation}\label{e.kernel}
\frac{\Lambda^{-1}}{|x-y|^{d+sp}} \leq \sfk (x,y,t)=\sfk(y,x,t) \leq \frac{\Lambda}{|x-y|^{d+sp}}\quad \mbox{a.e.} \quad  x,y \in \R^d
\end{equation}
for every $t \in (0,T)$ and for some constant $\Lambda \in [1,\infty)$. 

\subsubsection{Brief Summary}
Nonlocal problems, arise naturally in various contexts. These include jump processes (Lévy flights), nonlocal minimal surfaces, the Boltzmann equation, nonlocal kinetic equations, and the nonlocal analog of Hilbert's XIX problem. See \cite{X-RO} and its references for a brief summary. Harnack inequalities represent one of the essential cornerstones of this discipline. Indeed, the inherent roughness of the underlying data generally precludes higher-order regularity, hence the quantitative point-wise bound inferred by Harnack's inequality is, in a sense, the best regularity one can hope for.  Specifically, the aforementioned research focuses on nonlocal operators with bounded measurable coefficients of the form
\[
\cL u(x)=\mathrm{p.v.}\int_{\R^d}2|u(x)-u(y)|^{p-2}(u(x)-u(y))\sfk(x,y)\d{y}, 
\]
where $\sfk:\R^d \times \R^d \to [0,\infty)$ satisfies, for a.e. $x, y \in \R^d$, that
\begin{equation*}
\frac{\Lambda^{-1}}{|x-y|^{d+sp}} \leq \sfk(x,y)=\sfk(y,x) \leq \frac{\Lambda}{|x-y|^{d+sp}}
\end{equation*}
for some parameter $\Lambda \in [1,\infty)$, with $p >1$ and $s \in (0,1)$. Over the past decade, numerous authors have independently addressed the regularity of solutions of such equations. In fact, the elliptic case is hitherto well-understood; to give a representative example, ~\cite{Kas09} for (nonlocal) Moser's iteration method, whereas~\cite{CV10, Min11, DKP14, DKP16, PP22} for (nonlocal) De Giorgi's iteration methods. In stark contrast to the elliptic case, the parabolic case has still many open problems due to the long-range effects caused by time, despite the fact that the literature has grown considerably; see, for instance,~\cite{CCV11, FK13, Kim19, Str19, DKSZ20, APT22, BGK22, KW22b, Nak22a, Nak22b, CKW23, KW23, GK23, BK24, Pra24, Lia24a, Lia24b, SZ25, LW25, DeF26} and references therein. Recently, for the case $p=2$, Kassmann and Weidner provided a complete answer for the linear case in~\cite{KW24}, see also \cite{Lia25} for a different approach that can be applied to equations with drift terms.

\subsection{Statements of the main result}
We begin by introducing the minimal notation needed for the statements below. Throughout the present text, we work in dimension $d \geq 2$. For a parameter $a>0$, the symbol \[I_R^a (t_0):=\left(t_0-a R^{sp}, t_0+aR^{sp}\right)\]denotes the time-interval around $t_0$, of width $2a R^{sp}$. When $a=1$, we simply write it by $I_R(t_0)$. The notation $B_R(x_0)$, as usual, denotes the open Euclidean ball centered at $x_0$, of radius $R$.  When $x_0=0$ or it is clear from the context, we simply write $B_R$. The Lebesgue measure of a measurable set $U \subset \R^d$ by $|U|$ and the integral average of a function $g \in L^1U$ is denoted as usual with
\[
(g)_U:=\dashint_{U}g(x)\d{x}:=\frac{1}{|U|}\int_{U}g(x)\d{x}.
\] For fixed $p \in (1,\infty)$  and $s \in (0,1)$, we define the nowadays classic  \emph{nonlocal tail}, that will be used to describe the long-range interactions
\[
\tail\big(u\,; B_R(x_0)\big):=\left(R^{sp}\int_{\R^d \setminus B_R(x_0)}\frac{|u(x)|^{p-1}}{|x-x_0|^{d+sp}}\d{x} \right)^{\nicefrac{1}{(p-1)}}
\]
for $x_0 \in \R^d$ and $R>0$. First approaches to the H\"older regularity to variational nonlocal operators through a specific manipulation of $\tail$ can be found in~\cite{Kas09}, and then in \cite{DKP14} for a nonlinear version. Our first result is a $L^\infty$-$L^\nu$ quantitative local bound with optimal tail.
\begin{theorem}[Quantitative Local Bounds with Optimal Tail]\label{t.boundedness-NT}
For $p \in (1,\infty)$ and $s \in (0,1)$ let $u$ be a weak sub-solution (resp. super-solution) to~\eqref{e.NT}-\eqref{e.kernel} in the sense of Definition~\ref{def-of-NT}. Assume that $Q_{\rho, \theta}(z_0)  \subset \Omega_T$ for $z_0=(z_0,t_0) \in \Omega_T$. Let $\sigma \in (0,1)$ and $\nu \in (0,p]$. Then there exist positive constants $C(\data), C_{\nu}(\data, \nu)$, such that the following estimate holds true:
\begin{align*}
\sup_{\sigma Q_{\rho, \theta}(z_0)}u_\pm &\leq \frac{C}{(1-\sigma)^{(d+p)q_\ast+\frac{d+sp}{p-1}}}\left[ \frac{\theta}{(\sigma \rho)^{sp}}\dashint_{t_0-\theta}^{t_0}  \tail \big(u_\pm(t)\,; B_{\sigma \rho}(x_0)\big)^{p-1}\d{t}\right]^{\nicefrac{1}{(p-1)}}\\
&\quad \quad \quad \quad \quad \quad \quad \,+C_\nu \left[\frac{\boldsymbol{\cA}}{(1-\sigma)^{(d+p)q_\ast p}}\biint_{Q_{\rho,\theta}}u_\pm^\nu\dxt\right]^{\nicefrac{1}{\nu}},
\end{align*}
where we denoted
\[
\boldsymbol{\cA}:=\left(\dfrac{\rho^{sp}}{\theta}+\dfrac{1}{\sigma^{d+sp}}\right)^{pq_\ast}\left(\dfrac{\theta}{\rho^{sp}}\right) \quad \mbox{and} \quad
q_\ast:=\begin{cases}
\frac{1}{p}\left(\frac{2}{p \wedge 2}\frac{d}{sp}+1\right), \quad &\mbox{if}\quad sp<d,\\[2mm]
\frac{1}{p}\left(\frac{4}{p \wedge 2}+1\right), \quad &\textrm{if}\quad sp\geq d.
\end{cases}
\]
\end{theorem}
\noindent Our approach is based on~\cite[Section 3]{CN26}, comprehending the idea of the linear case as in~\cite{KW24} and ~\cite[Proposition 5.1]{Lia25}. See also the work of De Filippis~\cite{DeF26} for the fractional porous medium-type equation. Our second and main result (Theorem~\ref{t.fullHarnack-NT}) is the nonlocal strong Harnack inequality for local weak solutions to \eqref{e.NT}-\eqref{e.kernel}: a quantitative bound of the local supremum and the long-range effects of the positive part of the solution are quantitatively controlled by its infimum and the time-averaged tail of $u_{-}$. It can be compared to the nonlocal Harnack estimate for the nonlocal heat (or parabolic) equation; see~\cite{Kim19, Str19, KW24} for details. While classical Harnack inequalities focus purely on internal regularity, this theorem captures the essential interplay between local diffusion and long-range interactions within a single unified estimate.

\begin{theorem}[Strong Harnack inequality]\label{t.fullHarnack-NT} Let $u$ be a weak solution to~\eqref{e.NT}-\eqref{e.kernel} in the sense of Definition~\ref{def-of-NT} such that
\[
u \geq 0 \quad \mbox{in} \quad\cQ_{R_0}:=B_{R_0}(x_0) \times I_{R_0}(t_0)
\]
with $R_0:=4\cdot 6^{\nicefrac{1}{sp}}\rho$. There exists $C_\mathsf{H}>0$ depending only on the $\data$ such that, if condition \[\tail\big(u(\cdot); B_{R_0}(x_0) \big) \in L^{p-1+\eps}_{\mathrm{loc}}([0,T])\,\] is at stake for some $\varepsilon>0$, then, 
\begin{align*}
&\sup_{B_{\rho}(x_0) \times \left(t_0-\frac{1}{2}(2\rho)^{sp}, t_0\right]}u +\left(\dashint_{t_0-\frac{1}{2}(2\rho)^{sp}}^{t_0}\tail \big(u_+(t) ; B_\rho(x_0) \big)^{p-1}\d{t}\right)^{\frac{1}{p-1}} \notag\\[2pt]
& \quad \leq C_\mathsf{H}\left[ \inf_{B_\rho (x_0)\times \left(t_0+\frac{3}{4}(4\rho)^{sp}, t_0+(4\rho)^{sp} \right]}u +\left(\dashint_{I_{R_0}(t_0)} \tail \big(u_-(t)\,;B_{4\rho}(x_0)\big)^{p-1+\eps}\d{t}\right)^{\frac{1}{p-1+\eps}}\right],
\end{align*}
provided that the following inclusion is satisfied
\[
B_{2\rho}(x_0) \times \left(t_0-(2\rho)^{sp}, t_0+6(4\rho)^{sp}\right] \subset \cQ_{R_0} \Subset \Omega_T.
\]
\end{theorem}
\noindent If one assumes further $u$ globally nonnegative, the Harnack estimate above simplifies to
\[
\sup_{B_{\rho}(x_0) \times \left(t_0-\frac{1}{2}(2\rho)^{sp},t_0\right]}u
\leq C_\mathsf{H} \bigg( \inf_{B_\rho (x_0)\times \left(t_0+\frac{3}{4}(4\rho)^{sp}, t_0+(4\rho)^{sp} \right]}u\, \bigg)\,.
\] The time-gap phenomenon is unavoidable. Nevertheless, for {\it global} weak solutions the situation is completely different, see~\cite{LW25}
 .  

\subsection{Organization of the paper}
This paper is structured as follows: In Section~\ref{Sect.2} we collect several notation, definitions, and preliminary results.  Section~\ref{Sect.3} is devoted to the proof of Theorem~\ref{t.boundedness-NT}. To do so, we first study qualitative local boundedness (Proposition~\ref{t.qualitative-bnd-NT}). Section~\ref{Sect.4} is dedicated to measure theoretical properties (Critical Mass/De Giorgi type Lemmata, Expansion of Positivity) of the local solutions. At the end of this section, we give the proof of Theorem~\ref{t.fullHarnack-NT}.

\section{Notation and Preliminary Results}\label{Sect.2}
In this section we collect notation, definitions and preliminary results. Most of them are basically well-known in slightly different settings.

\subsection{Notation}
Throughout this paper, the symbols $C$ and $c$ denote positive constants which may
vary from line to line. For the sake of ease, we denote
\[
\data:=(d,s,p,\Lambda),
\]
where $\Lambda \in [1,\infty)$ is the structural constant appearing in~\eqref{e.kernel}. This allows us to denote a constant $C$ that depends on $(d,s,p,\Lambda)$ by simply $C(\data)$ instead of $C(d,s,p,\Lambda)$.
\smallskip

We denote $s \wedge k := \min\{s,k\}$ and $s \vee k:=\max\{s, k\}$. Moreover, $(s-k)_\pm:=\pm(s-k) \vee 0$. The symbol $\1_{A}$ denotes the indicator function on a set $A$. To lighten the notation, we denote $\sup_A \equiv \esssup_A$  and  $\inf_A \equiv \essinf_A$, respectively. 
\smallskip
  
Let $\Omega \subseteq \R^d$ be a bounded domain. For $T\in (0,\infty)$ we denote a space-time cylinder by $\Omega_T:=\Omega \times (0,T)$. As customary, $B_\rho(x_0)$ denotes the open ball with radius $\rho$ and center $x_0 \in \R^d$. For a fixed vertex $z_0=(x_0,t_0) \in \R^d \times \R$, we define the (one-sided) parabolic cylinder $Q_{\rho, \theta}(z_0)$ by
\[
Q_{\rho,\theta}(z_0):=B_\rho(x_0) \times (t_0-\theta, t_0]; 
\]
in particular, when $\theta=\rho^{sp}$ we will shorten $Q_{\rho}(z_0)\equiv Q_{\rho,\rho^{sp}}(z_0)$. When it is clear from the context which center is meant, we will omit it from the notation, that is, $B_\rho \equiv B_\rho(x_0)$ and $Q_{\rho,\theta}\equiv Q_{\rho,\theta}(z_0)$, etc. Throughout the paper, we fix the two-sided ambient cylinder
\[
\cQ_R(z_0):=B_{R}(x_0) \times (t_0-R^{sp},t_0+R^{sp}]=:B_R(x_0) \times I_R(t_0).
\]
\subsection{Function spaces}\label{Function-spaces}
A measurable function $u:\Omega \to \R$ is in the fractional Sobolev space $W^{s,p} (\Omega)$ if $u \in L^p(\Omega)$ and it has finite fractional norm
\[
\|u\|_{W^{s,p}(\Omega)}:=\|u\|_{L^p(\Omega)}+[u]_{W^{s,p}(\Omega)} <+\infty,
\]
where 
\[
[u]_{W^{s,p}(\Omega)}:=\left(\iint_{\Omega \times \Omega} \frac{|u(x)-u(y)|^p}{|x-y|^{d+sp}}\dxy \right)^{\nicefrac{1}{p}}
\]
is the \emph{Gagliardo} semi-norm. The fractional Sobolev space $W^{s,p}(\R^d)$ is defined analogously, and the fractional Sobolev space with zero boundary values is defined as
\[
W_{0}^{s,p}(\Omega):=\Big\{f \in W^{s,p}(\R^d): f=0\,\,\,\textrm{on}\,\,\,\R^d \setminus \Omega \Big\}.
\] Useful embeddings and properties of fractional Sobolev spaces can be found, for instance, in the comprehensive monographs~\cite{DNPV12, KP18, FeRo24, ADV25}. For what concerns the parabolic setting, we refer to the Bochner integral:  given $p \in [1,\infty)$,  $I \subset \R$ and an arbitrary Banach space $X$, we denote by $L^p(I ; X)$ the space of Lebesgue-measurable mappings $u : I \to X$ such that
\[
\|u\|_{L^p(I ; X)}:=\left(\int_I \|u(\cdot, t)\|_X^p\d{t} \right)^{\nicefrac{1}{p}}<\infty.
\]
Finally, we define $C(I; X)$ as the space of continuous maps $t \mapsto \|u(\cdot,t)\|_X$.

\subsection{Tools of the Trade}
In this subsection, we collect the technical tools that are used in the proofs. 

\subsubsection{Algebraic inequalities}
We begin with the useful algebraic inequalities. The first one is the following with constants being precise values, retrieved from~\cite[Lemma 2.2]{BDLMBS25} by careful inspections of the proof of \cite[Lemma 2.2]{AF89} for $\alpha \in (0,1)$ and \cite[inequality (2.4)]{GM86} for $\alpha \in (1,\infty)$.
\begin{lemma}\label{t.alg-est-1}
Let $\alpha>0$. Then
\[
C_1 \left(|a|+|b|\right)^{\alpha-1}|a-b| \leq \left||b|^{\alpha-1}b-|a|^{\alpha-1}a\right| \leq C_2 \left(|a|+|b|\right)^{\alpha-1}|a-b|
\]
holds whenever $a, b \in \R$, where 
\begin{equation*}
    C_1=
    \left\{
    \begin{array}{ll}
        \alpha, & \mbox{if $\alpha\in(0,1]$,} \\[2mm]
        2^{1-\alpha}, & \mbox{if $\alpha\in[1,\infty)$,}
    \end{array}
    \right.
    \qquad
    C_2=
    \left\{
    \begin{array}{ll}
         2^{1-\alpha}, & \mbox{if $\alpha\in(0,1]$,} \\[2mm]
         \alpha, & \mbox{if $\alpha\in[1,\infty)$.}
    \end{array}
    \right.
\end{equation*}
\end{lemma}

A little bit of algebra shows the following simple inequality.
\begin{lemma}\label{t.alg-est-2}
Let $\alpha>0$. Then, for every $a,b \in \R_{\geq 0}$.
\[
(a+b)^\alpha \leq 2^{(\alpha-1)_+} \left(a^\alpha+b^\alpha\right).
\]
\end{lemma}

Moreover, in order to dominate the auxiliary functions $\mathfrak{g}_\pm$ of the energy estimates, defined in~\eqref{e.def-of-g}, we will need the following two-sided bounds; the proof is in exactly the same analogy to~\cite[Lemma 2.2]{BDL21}.

\begin{lemma}\label{t.g}
There exist positive constants $c_1$, $c_2$, depending only on $p$, such that
\[
c_1\Big(|w|+|k(t)|\Big)^{p-2}(w-k(t))_\pm^2
 \leq \mathfrak{g}_\pm(w,k(t)) \leq c_2 \Big(|w|+|k(t)|\Big)^{p-2}(w-k(t))_\pm^2.
 \]
\end{lemma}



The next lemma that is needed in order to run the Moser's iteration machinery. The precise proof is in~\cite[Lemma A.6]{CN26}.
\begin{lemma}\label{t.useful}
For $p \in (1,\infty)$ and $\alpha > -1$ we have the two-sided bound
\[
\frac{1}{2(\alpha+p+1)^{p+2}} \leq \int_0^1 \lambda^\alpha(1-\lambda)^p \d{\lambda} \leq \frac{2}{\alpha+p+1}.
\]
\end{lemma} 

Finally, we recall the well-known~\emph{fast geometric convergence}. A proof can be found in~\cite[Chapter I.4, Lemma 4.1]{DiB93}.
\begin{lemma}\label{t.FGC}
Let $\{\sfY_i\}_{i \in \N_0}$ be a sequence of positive numbers satisfying recursive inequalities
\begin{equation*}
\sfY_{i+1} \leq C\,\sfb^i\sfY_i^{1+\beta},
\end{equation*}
where $C>0$, $\sfb>1$ and $\beta>0$ are given constants independent of $i \in \N_0$. Then
$\sfY_i \to 0$ as $i\to \infty$, provided that $ \sfY_0 \leq C^{-1/\beta}\sfb^{-1/\beta^2}$.
\end{lemma}

\subsubsection{Functional inequalities}
Firstly, we state the following parabolic version of fractional Sobolev inequality, whose proof is similar to~\cite[Propositions A.2 and A.3]{Lia24a}, with a different exponent $\kappa$, since here we collect the supremum over time of the spatial $L^p$ norm of $u$.
\begin{proposition}\label{FS}
Let $p \geq 1$, $s \in (0,1)$ and set
\[
\kappa_\ast:=\begin{cases}
\frac{d}{d-sp} \quad &\textrm{if}\quad sp<d,\\
2 \quad &\textrm{if}\quad sp\geq d.
\end{cases}
\]
For every function $w \in L^\infty\left(t_1,t_2\,; L^p(B_R)\right) \cap L^p(t_1,t_2\,;W^{s,p}(B_R))$
which is compactly supported in $B_{(1-\mathsf{d})R}$ for some $\mathsf{d} \in (0,1)$ and for a.e. $t \in (t_1,t_2)$, we have
\begin{align*}
&\int_{t_1}^{t_2}\int_{B_R}|w|^{\kappa p}\dxt \\
& \quad \leq C\Bigg[R^{sp}\int_{t_1}^{t_2}\iint \nolimits_{B_R \times B_R}\frac{|w(x,t)-w(y,t)|^p}{|x-y|^{d+sp}}\dxyt+\frac{1}{\mathsf{d}^{d+sp}}\int_{t_1}^{t_2}\int_{B_R}|w|^p\dxt\Bigg] \\
&\quad \quad \quad \times \left(\sup_{t_1<t<t_2}\dashint_{B_R}|w(t)|^p\d{x}\right)^{\frac{\kappa_\ast-1}{\kappa_\ast}}
\end{align*}
with $\kappa:=1+\frac{\kappa_\ast-1}{\kappa_\ast}$, where $C=C(s,p,d)>0$. 
\end{proposition}

We provide a Giagliardo--Nirenberg type estimate, retrieved from~\cite[Lemma 2.2]{BK24}.

\begin{lemma}\label{t.GN}
Let $1 \leq m \leq (p \vee 2)$. Suppose that the positive exponents $\widetilde{q}$ and $\widetilde{r}$ satisfy
\[
\frac{1}{\widetilde{q}}+\frac{1}{\widetilde{r}}\left(\frac{sp}{d}+\frac{p}{m}-1\right)=\frac{1}{m},
\]
where they obey
\[
\left\{
    \begin{array}{lll}
    \widetilde{q} \in \left[m, \frac{dp}{d-sp}\right], \quad &\widetilde{r} \in [p,\infty),  & \mbox{if $d>sp$,} \\[3mm]
     \widetilde{q} \in [m, \infty), \quad &\widetilde{r} \in \left(\frac{msp}{d}+p-m,\infty\right),  & \mbox{if $d=sp$,} \\[3mm]
      \widetilde{q} \in [m, \infty], \quad &\widetilde{r} \in \left[\frac{msp}{d}+p-m,\infty\right),  & \mbox{if $d<sp$.}
    \end{array}
    \right.
\]
There exists a constant $C(d,s,p,\widetilde{q},\widetilde{r})<\infty$ such that
\begin{align*}
\left(\int_{t_1}^{t_2} \|f(t)\|_{L^{\widetilde{q}}(B_R)}^{\widetilde{r}}\d{t}\right)^{\frac{d\widetilde{q}}{sp\widetilde{q}+d\widetilde{r}}} &\leq C \left(\int_{t_1}^{t_2}[f(t)]_{W^{s,p}(B_R)}^p\d{t} \right.\\
& \left. \quad \quad \quad +R^{-sp} \int_{t_1}^{t_2} \|f(t)\|_{L^p(B_R)}^p\d{t} 
+\sup_{t_1<t<t_2}\|f(t)\|_{L^m(B_R)}^m\right)
\end{align*}
whenever the right side is finite.
\end{lemma}
Finally, we record an inequality for the Gagliardo semi-norm, whose proof can be extracted from~\cite[Lemma 2.3]{BK24}.
\begin{lemma}\label{t.alg-est-3}
Let $u \in W^{s,p}(B_R)$. Then for any $0\leq b \leq a$, we have
\[
\left[(u-b)_-\right]_{W^{s,p}(B_R)}^p \leq \left[(u-a)_-\right]_{W^{s,p}(B_R)}^p.
\]
and
\[
\left[(u-a)_+\right]_{W^{s,p}(B_R)}^p \leq \left[(u-b)_+\right]_{W^{s,p}(B_R)}^p.
\]
\end{lemma}

\subsection{Weak formulation}
We state the definition of local weak solutions to~\eqref{e.NT}-\eqref{e.kernel}. To ease notation, we denote by $J_p(a):=|a|^{p-2}a$ for $a \in \R$.

\begin{definition}\label{def-of-NT}
A measurable function $u: \R^d\times (0,T) \rightarrow \R$ is called a \emph{weak sub(super)-solution} to~\eqref{e.NT} provided that the following conditions are satisfied: 
\begin{itemize}
\item $u \in L^\infty(0,T ; W^{s,p}(\R^d)) \cap C\left([0,T] ; L^p(\Omega)\right)$,
\item the energy identity 
\begin{align}\label{e.weak-formulation}
&-\iint_{\Omega_T}|u|^{p-2}u\cdot \partial_t\phi\dxt \notag\\
&\quad +\int_0^T\iint_{\R^d \times \R^d}J_p(u(x,t)-u(y,t))(\phi(x,t)-\phi(y,t))\sfk(x,y,t)\dxyt \leq (\geq ) 0
\end{align}
is valid for any test function $0 \leq \phi \in \mathscr{T}$, where the class of testing functions $\mathscr{T}$ is given by
\[
\mathscr{T}:=\left\{\varphi \in L^p\left(0,T\,;W_{0}^{s,p}(\Omega)\right) \cap W^{1,p}\left(0,T\,; L^{p}(\Omega)\right) \,\,\,\Bigg|\, \begin{array}{c}\varphi(x,0)=\varphi(x,T)=0 \\\textrm{for a.e.} \,\,x \in \Omega\end{array}\right\}.
\]
\end{itemize}
We call $u$ a (local) \emph{weak solution} provided that $u$ is both a weak sub-solution and a weak super-solution.
\end{definition}

\begin{remark}
Even under the minimal assumption that $u \in  L^\infty(0,T ; W^{s,p}(\R^d))$ the $L^p$-continuity follows a-priori. Indeed, a rigorous analysis using exponential mollification establishes this fact;  we refer to~\cite[footnotes, pages 15--16]{MNY23} for an in-depth discussion.
\end{remark}

\subsection{Caccioppoli inequalities}
 Here we show several energy estimates embodied by local weak solutions. Let us introduce the function
\begin{equation}\label{e.def-of-g}
\mathfrak{g}_\pm(w,k(t)):=\pm(p-1)\int_{k(t)}^w|\tau|^{p-2}(\tau-k(t))_\pm\d{\tau},
\end{equation}
where $k(t)$ is an arbitrary absolutely continuous function on $(0,T)$. Function $\mathfrak{g}_\pm(w,k(t))$ is a purely doubly nonlinear term: it is the main obstacle to regularity, since it prevents, for instance, the use of Harnack inequality for a proof of H\"older continuity. See \cite{CHYSS} for instance for an overview. 
\begin{remark}
By virtue of the fundamental theorem of calculus for the bivariate function after changing variables (see~\cite[(3.5), page 133]{FMS}) and the monotonicity of $\tau \mapsto |\tau|^{\frac{1}{p-1}-1}$, we have
\begin{align}\label{e.fct-on-g}
\partial_t\mathfrak{g}_\pm(w,k(t))&=\pm \partial_t \left[\int_{|k(t)|^{p-2}k(t)}^{|w|^{p-2}w}\left(|\tau|^{\frac{1}{p-1}-1}\tau -k(t)\right)_\pm \d{\tau}\right] \notag\\[3mm]
&=\pm \int_{|k(t)|^{p-2}k(t)}^{|w|^{p-2}w}\partial_t\left(|\tau|^{\frac{1}{p-1}-1}\tau -k(t)\right)_\pm \d{\tau}\pm \partial_t\left(|w|^{p-2}w\right)\left(w-k(t)\right)_\pm \notag\\[3mm]
&=\mp k^\prime (t) \big(|w|^{p-2}w-|k(t)|^{p-2}k(t)\big)_\pm \pm \partial_t\left(|w|^{p-2}w\right)\left(w-k(t)\right)_\pm
\end{align}
with upper or lower sign taken together. Here the integrand $|\tau|^{\frac{1}{p-1}-1}\tau$ is understood as $0$ when $\tau=0$.  
\end{remark}

The following Caccioppoli estimates have a time-dependent truncation level $k(t)$. The idea is basically taken from~\cite{KW24, Lia24b}.
\begin{proposition}\label{t.caccioppoli1}
Let $u$ be  a weak sub(super)-solution to~\eqref{e.NT}-\eqref{e.kernel} in the sense of Definition~\ref{def-of-NT}. Let us fix a cylinder $Q_{\rho,\tau}(z_0) \Subset \Omega_T$ for $z_0=(x_0,t_0) \in \Omega_T$. Let $\zeta$ be a nonnegative smooth cutoff function such that $\zeta(\cdot,t)$ is compactly supported in $B_\rho(x_0)$ for all $t \in (t_0-\tau,t_0)$.  There exists $C(p,\Lambda)>0$ such that

\begin{align}\label{e.caccioppoli1}
&\sup\limits_{t_0-\tau<t<t_0}\int_{B_\rho(x_0)}\zeta^p\mathfrak{g}_\pm(u,k(t))\d{x} \notag\\[4pt]
&\quad \quad \quad +\int_{t_0-\tau}^{t_0}\iint_{B_\rho(x_0) \times B_\rho(x_0)}\left(\zeta^p(x,t) \wedge  \zeta^p(y,t)\right) \dfrac{\big|w_{\pm}(x,t)-w_{\pm}(y,t)\big|^p}{|x-y|^{d+sp}}\dxyt  \notag \\[4pt]
& \quad \quad \quad +\iint_{Q_{\rho,\tau}(z_0)}\zeta^pw_\pm(x,t)\d{x}\left(\int_{\R^d}\frac{w_\mp^{p-1}(y,t)}{|x-y|^{d+sp}}\d{y}\right)\d{t} \notag\\[4pt]
& \quad \quad \leq \int_{B_\rho(x_0) \times \{t_0-\tau\}}\zeta^p\mathfrak{g}_\pm(w,k(t))\d{x}+C\iint_{Q_{\rho,\tau}(z_0)}\left|\partial_t \zeta^p\right|\mathfrak{g}_\pm(u,k(t))\dxt \notag\\[4pt]
&\quad \quad \quad +C\int_{t_0-\tau}^{t_0}\iint_{B_\rho(x_0)\times B_\rho(x_0)}\left(w_\pm^p(x,t) \wedge  w_\pm^p(y,t)\right)\dfrac{\big|\zeta(x,t)-\zeta(y,t)\big|^p}{|x-y|^{d+sp}}\dxyt \notag\\[4pt]
&\quad \quad \quad +C\iint_{Q_{\rho,\tau}(z_0)}\zeta^pw_\pm(x,t)\d{x} \left(\sup_{x\,\in\, \supp \zeta(\cdot,t)}\int_{\R^d \setminus B_\rho(x_0)}\frac{w_\pm^{p-1}(y,t)}{|x-y|^{d+sp}}\d{y}\right)\d{t} \notag\\[4pt]
&\quad \quad  \quad \mp C\iint_{Q_{\rho,\tau}(z_0)}k^\prime(t)\zeta^p\Big(|u|+|k(t)|\Big)^{p-2}w_\pm(x,t)\dxt,
\end{align}
for any arbitrary absolutely continuous function $k:(0,T) \rightarrow \R$, where we used $w_\pm(x,t):=\left(u(x,t)-k(t)\right)_\pm$ to shorten the notation. 
\end{proposition}

\begin{proof} The fractional term on the left side of~\eqref{e.caccioppoli1} can be dealt with a completely similar argument to~\cite[Lemma 3.23]{CCMV25}, see also \cite{Coz17}. On the other hand, the last term on the right side of~\eqref{e.caccioppoli1} is obtained using identity~\eqref{e.fct-on-g}; a routine computation says, informally, that
\begin{align*}
\iint_{\Omega_T}\partial_t&\Big(|u|^{p-2}u\Big)(u-k(t))_+\zeta^p\dxt \\
&\geq \iint_{\Omega_T}\partial_t\mathfrak{g}_+(u,k(t))\zeta^p\dxt+\iint_{\Omega_T}k^\prime(t)\Big(|u|^{p-2}u-|k(t)|^{p-2}k(t)\Big)_+\zeta^p\dxt\\
&\gtrsim_p \iint_{\Omega_T}\partial_t\mathfrak{g}_+(u,k(t))\zeta^p\dxt+\iint_{\Omega_T}k^\prime(t)\Big(|u|+|k(t)|\Big)^{p-2}(u-k(t))_+\zeta^p\dxt.
\end{align*}
The remaining part is exactly the same as~\cite[Appendix B]{Nak23}. The rigorous proof relies on ``exponential mollification'' techniques. This is first adopted for the doubly nonlinear problems~\cite{BDL21}; in particular, for the nonlocal case, the reader can also consult the arguments of~\cite[Appendix A]{Nak22a} for more details of this point.
\end{proof}

By tailoring the cutoff function to the geometry of the cylinders, a direct application of Proposition~\ref{t.caccioppoli1} yields the following Caccioppoli estimate.

\begin{proposition}\label{t.caccioppoli2}
Let $u$ be  a weak super-solution to~\eqref{e.NT}-\eqref{e.kernel} in the sense of Definition~\ref{def-of-NT}. Let us fix a cylinder $Q_{\rho,\tau}(z_0) \Subset \Omega_T$ for $z_0=(x_0,t_0) \in \Omega_T$. There exists a constant $C(\data)>0$ such that
\begin{align}\label{e.caccioppoli2}
&\sup\limits_{t_0-\tau_2<t<t_0}\int_{B_{r_1}(x_0)\times \{t\}}\mathfrak{g}_\pm(u,k)\d{x}\notag\\[4pt]
&\quad \quad+\int_{t_0-\tau_2}^{t_0}\iint_{B_{r_1}(x_0)\times B_{r_1}(x_0)}\dfrac{\big|w_{\pm}(x,t)-w_{\pm}(y,t)\big|^p}{|x-y|^{d+sp}}\dxyt \notag \\[4pt]
&\quad \quad+\iint_{Q_{r_1,\tau_1}(z_0)}w_\pm(x,t)\left(\int_{\R^d}\dfrac{w_\mp^{p-1}(y,t)}{|x-y|^{d+sp}}\d{y}\right)\dxt \notag \\[4pt]
&\quad \leq \int_{B_{r_1}(x_0)\times \{ t_0-\tau_2\}}\mathfrak{g}_\pm(u,k)\d{x}+C\frac{r_2^{(1-s)p}}{(r_2-r_1)^p}\iint_{Q_{r_2,\tau_2}(z_0)}w_\pm^p\dxt \notag \\[4pt]
&\quad \quad +C\frac{r_2^d}{(r_2-r_1)^{d+p}}\iint_{Q_{r_2,\tau_2}(z_0)}w_{\pm}(x,t)\d{x} \cdot \tail \big(w_\pm(t)\,; B_{r_2}(x_0)\big)^{p-1}\d{t},
\end{align}
and
\begin{align}\label{e.caccioppoli3}
&\sup\limits_{t_0-\tau_2<t<t_0}\int_{B_{r_1}(x_0)\times \{t\}}\mathfrak{g}_\pm(u,k)\d{x}\notag\\[4pt]
&\quad \quad+\int_{t_0-\tau_2}^{t_0}\iint_{B_{r_1}(x_0)\times B_{r_1}(x_0)}\dfrac{\big|w_{\pm}(x,t)-w_{\pm}(y,t)\big|^p}{|x-y|^{d+sp}}\dxyt \notag \\[4pt]
&\quad \quad+\iint_{Q_{r_1,\tau_1}(z_0)}w_\pm(x,t)\left(\int_{\R^d}\dfrac{w_\mp^{p-1}(y,t)}{|x-y|^{d+sp}}\d{y}\right)\dxt \notag \\[4pt]
&\quad \leq \frac{C}{\tau_2-\tau_1} \iint_{Q_{r_2, \tau_2}(z_0)}\mathfrak{g}_\pm(u,k)\dxt+C\frac{r_2^{(1-s)p}}{(r_2-r_1)^p}\iint_{Q_{r_2,\tau_2}(z_0)}w_\pm^p\dxt\notag \\[4pt]
&\quad \quad +C\frac{r_2^{d}}{(r_2-r_1)^{d+sp}}\iint_{Q_{r_2,\tau_2}(z_0)}w_{\pm}(x,t)\d{x} \cdot \tail \big(w_\pm(t)\,; B_{r_2}(x_0) \big)^{p-1}\d{t}
\end{align}
for any concentric cylinders $Q_{r_1,\tau_1} (z_0)\subset Q_{r_2,\tau_2}(z_0) \subset \cQ_R$ and any time-independent level $k \in \R$.
\end{proposition}

\begin{proof}
The proof is exactly same as~\cite[Appendix B]{Nak23}, selecting an appropriate smooth cutoff function in~\eqref{e.caccioppoli1} and taking $k(t)=k \in \R$.
\end{proof}

\section{Refined Upper Bounds: Proof of Theorem~\ref{t.boundedness-NT}}\label{Sect.3}
We give a two-step proof. First we deduce the qualitative $L^\infty$-$L^\nu$ local boundedness for weak sub-solutions to~\eqref{e.NT} under optimal Tail condition. Our approach is based on~\cite{CN26}. What is new in our approach is that we are able to prove these estimates without using strict positivity information about the solution; see~\cite[Theorem 2.15]{BGK22} and~\cite[Theorems 1.1 and 1.2]{Nak22a} for more details.

\begin{proposition}[Qualitative local boundedness]\label{t.qualitative-bnd-NT}
For $p \in (1,\infty)$ and $s \in (0,1)$, let $u$ be a weak sub-solution to~\eqref{e.NT}-\eqref{e.kernel} in the sense of Definition~\ref{def-of-NT}. Suppose that $u \in L^{p-1+\eps}_{\mathrm{loc}}(\Omega_T)$ for $\eps \in (0,\infty)$. Then, $u$ is locally bounded in $\Omega_T$. More precisely, there exists a constant $C(\data, \eps)>0$ such that on every concentric cylinders $\sigma^\prime Q_{R,S}(z_0) \subset \sigma Q_{R,S}(z_0) \subset \Omega_T$ for $z_0=(x_0,t_0) \in \Omega_T$ and $0<\sigma^\prime < \sigma \leq 1$, we have
\begin{align*}
\sup_{\sigma^\prime Q_{R,S}(z_0)} u_+ &\leq \frac{C}{(\sigma-\sigma^\prime)^{\frac{d+p}{(\kappa-1)(p-1+\eps)}}}  \left[\left(\frac{S}{R^{sp}}+1 \right)+\frac{S}{R^{sp+\frac{d(p-1)}{p-1+\eps}}} \widetilde{\textbf{\textsf{Tail}}}\right]^{\frac{\kappa}{(\kappa-1)(p-1+\eps)}} \\[4pt]
& \quad \quad \quad \times \left(\biint_{\sigma Q_{R,S}(z_0)}\left[1+u_+^{p-1+\eps} \right]\dxt \right)^{\nicefrac{1}{(p-1+\eps)}},
\end{align*}
where $\kappa:=1+\frac{\kappa_\ast-1}{\kappa_\ast}$ with $\kappa_\ast$ being the Sobolev embedding exponent as in Proposition~\ref{FS}, and to lighten the notation we denoted $\sigma Q_{R,S}(z_0):=Q_{\sigma R, \sigma S}(z_0)$ and
\[
\widetilde{\textbf{\textsf{Tail}}}:=\left(\dashint_{t_0-\sigma S}^{t_0} \tail \big(u_+(t)\,; B_{\sigma^\prime R}(x_0) \big)^{p-1+\eps}\d{t}\right)^{\frac{p-1}{p-1+\eps}}.
\]
\end{proposition}

\begin{remark}
The proof of Proposition~\ref{t.qualitative-bnd-NT} is based on Moser's iteration presented in~\cite[Proposition 3.1]{CN26}. Our approach can remove the positivity of solutions to~\eqref{e.NT}. See also~\cite[Section 5]{BGK22} or~\cite[Section 4]{Nak22b} for more thorough discussion.
\end{remark}

\begin{proof}
We will follow the argument given in~\cite[Section 3.1]{CN26}. Without loss of generality, we let $z_0=(x_0,t_0)=(0,0)$. 
\smallskip

\emph{Step 1.} We start from~\eqref{e.caccioppoli3}$_+$ with $r_1=(R+\rho)/2$, $r_2=R$, $\tau_1=\tau$ and $\tau_2=S$. After multiplying it by $k^\alpha$ for $\alpha>-1$, we integrate the resulting inequality over $k \in [0,\infty)$ to obtain
\begin{align}\label{e.quali-bnd-1-NT}
\int_0^\infty &k^\alpha\int_{B_{\frac{R+\rho}{2}}}\mathfrak{g}_+(u(t),k)\dxk\notag \\[4pt]
&\leq \frac{C}{S-\tau}\int_0^\infty k^\alpha \iint_{Q_{R,S}}\mathfrak{g}_+(u,k)\dxtk \notag\\[4pt]
&\quad \quad+ C\frac{2^pR^{(1-s)p}}{(R-\rho)^p}\int_0^\infty k^\alpha \iint_{Q_{R,S}}(u-k)_+^p\dxtk \notag \\[4pt]
&\quad \quad +C\frac{2^{d+sp}R^{d}}{(R-\rho)^{d+sp}}\int_0^\infty k^\alpha \iint_{Q_{R,S}}(u-k)_+\d{x} \cdot \tail \big((u-k)_+(t)\,; B_{R} \big)^{p-1}\dtk \notag \\[4pt]
&=:\sfI+\sfI\sfI+\sfI\sfI\sfI.
\end{align}
Appealing to Fubini's theorem and H\"{o}lder's inequality after a change of variables, we straightforwardly estimate that
\begin{align*}
\sfI \sfI &\leq \frac{C}{\alpha+p+1} \frac{R^{(1-s)p}}{(R-\rho)^p} \iint_{Q_{R,S}} u_+^{\alpha+p+1}\dxt,\\[4pt]
\sfI\sfI \sfI&\leq \frac{C}{\alpha+p+1} \frac{R^d}{(R-\rho)^{d+sp}}\left(\int_{-S}^0 \tail \big(u_+(t) ; B_R \big)^{p-1+\eps}\d{t}\right)^{\frac{p-1}{p-1+\eps}}\notag\\
&\quad \quad \quad \quad \quad \quad \quad \quad  \quad \quad \times \left(\iint_{Q_{R,S}}u_+^{(\alpha+2)\frac{p-1+\eps}{\eps}}\dxt\right)^{\frac{\eps}{p-1+\eps}}.
\end{align*}
We now turn our attention to $\sfI$. For this, we need to consider two cases: $p \geq 2$ and $1<p<2$. In the first case $p \geq 2$, Lemma~\ref{t.g}, using Lemma~\ref{t.useful} with $p=2$, implies that
\begin{align*}
\sfI &\leq \frac{C}{S-\tau}\int_0^\infty k^\alpha \iint_{Q_{R,S}}\Big(|u|+|k|\Big)^{p-2}(u-k)_+^2\dxtk \\
&\leq \frac{C}{S-\tau}\left(\iint_{Q_{R,S}}u_+^{\alpha+p+1}\dxt \right)\left(\int_0^1\lambda^\alpha (1+\lambda)^{p-2}(1-\lambda)^2\d{\lambda}\right)\\
&\leq \frac{C}{S-\tau}\left(\iint_{Q_{R,S}}u_+^{\alpha+p+1}\dxt\right) \cdot \frac{2}{\alpha+3}.
\end{align*}
In the latter case, since $|u|+|k| \geq (u-k)_+$, using Lemma~\ref{t.useful} again, we have
\begin{align*}
\sfI &\leq \frac{C}{S-\tau} \int_0^\infty k^\alpha \iint_{Q_{R,S}}(u-k)_+^p\dxtk\\
&\leq \frac{C}{S-\tau}\left(\iint_{Q_{R,S}}u_+^{\alpha+p+1}\dxt\right) \left(\int_0^1\lambda^\alpha (1-\lambda)^p\d{\lambda}\right)\\
&\leq \frac{C}{S-\tau}\left(\iint_{Q_{R,S}}u_+^{\alpha+p+1}\dxt\right) \cdot \frac{2}{\alpha+p+1}.
\end{align*}
Altogether, we have shown that
\[
\sfI \leq \frac{1}{\alpha+(p\vee 2)+1}\cdot \frac{C}{S-\tau}\left(\iint_{Q_{R,S}}u_+^{\alpha+p+1}\dxt\right).
\]

Next, we handle the left side of~\eqref{e.quali-bnd-1-NT}. Using Lemmas~\ref{t.g} and~\ref{t.useful} again, we deduce that, for every $t \in (-S,0)$, 
\[
\int_0^\infty k^\alpha\int_{B_{\frac{R+\rho}{2}}}\mathfrak{g}_+(u(t),k)\dxk \geq \frac{C}{(\alpha+( p \wedge 2)+1)^{(p \wedge 2)+2}}\int_{B_{\frac{R+\rho}{2}}}u_+^{\alpha+p+1}(t)\d{x}.
\]
Combining this with the estimates of $\sfI$--$\sfI\sfI\sfI$ and rearranging, we obtain
\begin{align}\label{e.quali-bnd-2-NT}
\sup_{t \in (-S, \,0)}&\,\dashint_{B_{\frac{R+\rho}{2}}}u_+^{\alpha+p+1}(t)\d{x}  \notag \\[4pt]
&\leq C\frac{(\alpha+( p \wedge 2)+1)^{(p \wedge 2)+2}}{\alpha+(p \vee 2)+1}\left\{\left[\frac{R^{(1-s)p}}{(R-\rho)^p} +\frac{1}{S-\tau}\right] \frac{|Q_{R,S}|}{|B_{\frac{R+\rho}{2}}|}\biint_{Q_{R,S}} u_+^{\alpha+p+1}\dxt \right.\notag\\[4pt]
& \quad \quad \quad \quad\left.+\frac{R^{d}}{(R-\rho)^{d+sp}} \frac{|Q_{R,S}|^{\frac{\eps}{p-1+\eps}}}{|B_{\frac{R+\rho}{2}}|}\left(\int_{-S}^0 \tail \big(u_+(t)\,; B_{R} \big)^{p-1+\eps}\d{t}\right)^{\frac{p-1}{p-1+\eps}} \right. \notag\\[4pt]
&\quad \quad \quad \quad \quad \quad \quad \quad \quad  \left.\times \left(\biint_{Q_{R,S}}u_+^{(\alpha+2)\frac{p-1+\eps}{\eps}}\dxt\right)^{\frac{\eps}{p-1+\eps}} \right\}\notag\\[4pt]
&\leq C\frac{(\alpha+( p \wedge 2)+1)^{(p \wedge 2)+2}}{\alpha+(p \vee 2)+1}\left\{\left[\frac{R^{(1-s)p}}{(R-\rho)^p} +\frac{1}{S-\tau}\right] \frac{|Q_{R,S}|}{|B_{\frac{R+\rho}{2}}|}\biint_{Q_{R,S}} u_+^{\alpha+p+1}\dxt \right.\notag\\[4pt]
& \quad \quad \quad \quad\left.+\frac{R^{d}}{(R-\rho)^{d+sp}} \frac{|Q_{R,S}|^{\frac{\eps}{p-1+\eps}}}{|B_{\frac{R+\rho}{2}}|}\left(\int_{-S}^0 \tail \big(u_+(t)\,; B_{R} \big)^{p-1+\eps}\d{t}\right)^{\frac{p-1}{p-1+\eps}} \right. \notag\\[4pt]
&\quad \quad \quad \quad \quad \quad \quad \quad \quad  \left.\times \left(\biint_{Q_{R,S}}\left[1+u_+^{(\alpha+2)\frac{p-1+\eps}{\eps}} \right] \dxt \right) \right\}
\end{align}
for a constant $C(\data)<\infty$, where the subsequent inequality was used to get the last line
\[
\left(\biint_{Q_{R,S}}u_+^{(\alpha+2)\frac{p-1+\eps}{\eps}}\dxt \right)^{\frac{\eps}{p-1+\eps}} \leq \biint_{Q_{R,S}}\left[ 1+u_+^{(\alpha+2)\frac{p-1+\eps}{\eps}} \right] \dxt.
\]
Similarly, for $\eps >1$, we select $\alpha \in (-1,\infty)$ which satisfies
\[
\alpha+p+1=(\alpha+2)\frac{p-1+\eps}{\eps} \quad \iff \quad \alpha=\eps-2,
\]
and hence
\begin{align}\label{e.quali-bnd-3-NT}
\sup_{t \in (-S, \,0)} &\,\dashint_{B_{\frac{R+\rho}{2}}}u_+^{p-1+\eps}(t)\d{x}  \notag \\[4pt]
&\leq C\Gamma_{p,\eps}\left\{\left[\frac{R^{(1-s)p}}{(R-\rho)^p} +\frac{1}{S-\tau}\right] \frac{|Q_{R,S}|}{|B_{\frac{R+\rho}{2}}|}+\frac{R^{d}}{(R-\rho)^{d+sp}} \frac{|Q_{R,S}|^{\frac{\eps}{p-1+\eps}}}{|B_{\frac{R+\rho}{2}}|}\sfT_{R,S}\right\}  \notag\\[4pt]
&\quad \quad \quad \quad \quad \quad \quad \quad \quad  \times \left(\biint_{Q_{R,S}} \left[1+u_+^{p-1+\eps} \right]\dxt\right),
\end{align}
where we denoted
\[
\Gamma_{p,\eps}:=\frac{\left(( p \wedge 2)+\eps-1\right)^{(p \wedge 2)+2}}{(p \vee 2)+\eps-1}\quad \mbox{and}\quad \sfT_{R,S}:=\left(\int_{-S}^0 \tail \big(u_+(t)\,; B_{R} \big)^{p-1+\eps}\d{t}\right)^{\frac{p-1}{p-1+\eps}}
\]
to lighten the notation. 
\smallskip

\emph{Step 2.} Once we have~\eqref{e.quali-bnd-3-NT}, a reverse H\"{o}lder type inequality follows in a straightforward way. Indeed, we denote $w:=u_+^{\nicefrac{(p-1+\eps)}{p}}$ and 
\[
\widetilde{Q}:=Q_{\frac{R+\rho}{2}, \frac{S+\tau}{2}}=\widetilde{B}\times \widetilde{T}, \quad \widetilde{B}:=B_{\frac{R+\rho}{2}}, \quad \widetilde{T}:=\left(-\tfrac{S+\tau}{2}, 0\right).
\]
Let $\varphi$ be a cutoff function in $Q_{R,S}$ satisfying
\[
\1_{Q_{\rho,\tau}} \leq \varphi \leq \1_{\widetilde{Q}}, \quad |\nabla \varphi| \lesssim 1/(R-\rho), \quad \supp \varphi \subset \widetilde{Q}.
\]
Appealing to Proposition~\ref{FS} with $\mathsf{d}=\nicefrac{(R-\rho)}{2R} \in (0,1)$ we obtain that, for $\kappa=1+\frac{\kappa_\ast-1}{\kappa_\ast}$, 
\begin{align}\label{e.quali-bnd-4-NT}
\biint_{Q_{\rho, \tau}}&w^{\kappa p}\dxt \leq \frac{1}{|Q_{\rho, \tau}|}\iint_{\widetilde{Q}}(\varphi w)^{\kappa p}\dxt \notag\\[4pt]
&\leq \frac{C}{|Q_{\rho, \tau}|}\left[ \left(\frac{R+\rho}{2}\right)^{sp} \int_{\widetilde{T}} \iint \nolimits_{\widetilde{B} \times \widetilde{B}} \frac{|\varphi w(x,t)-\varphi w(y,t)|^p}{|x-y|^{d+sp}}\dxyt \right.\notag\\
&\quad \quad \quad \quad \left.+\left(\frac{2R}{R-\rho}\right)^{d+sp}\iint_{\widetilde{Q}}(\varphi w)^p\dxt\right] \left[\sup_{t \in \widetilde{T}} \dashint_{\widetilde{B}}(\varphi w)^p(t)\d{x} \right]^{\frac{\kappa_\ast-1}{\kappa_\ast}}\notag\\[4pt]
&\leq C\frac{|Q_{R,S}|}{|Q_{\rho, \tau}|} \left(\frac{2R}{R-\rho}\right)^{d+sp} \left[\biint_{Q_{R,S}} w^p\dxt+\left(\sup_{t \in (-S,0)} \dashint_{B_{\frac{R+\rho}{2}}}w^p(t)\d{x} \right)\right]^\kappa,
\end{align}
where we used
\[
\int_{\widetilde{B}} \frac{\d{y}}{|x-y|^{d+sp}} \leq C(d,s,p) \left(\frac{R+\rho}{2}\right)^{-sp}, \quad \forall x \in \widetilde{B}\equiv B_{\frac{R+\rho}{2}}.
\]
The combination of~\eqref{e.quali-bnd-3-NT} and~\eqref{e.quali-bnd-4-NT}  yields that
\begin{align}\label{e.quali-bnd-5-NT}
\biint_{Q_{\rho, \tau}}&w^{\kappa p}\dxt \leq C\Gamma_{p,\eps}^{(p+1)\kappa} \frac{|Q_{R,S}|}{|Q_{\rho, \tau}|} \left(\frac{2R}{R-\rho}\right)^{d+sp} \notag \\[4pt]
&\times \left\{\left[\frac{R^{(1-s)p}}{(R-\rho)^p} +\frac{1}{S-\tau}\right] \frac{|Q_{R,S}|}{|B_{\frac{R+\rho}{2}}|}+\frac{R^{d}}{(R-\rho)^{d+sp}} \frac{|Q_{R,S}|^{\frac{\eps}{p-1+\eps}}}{|B_{\frac{R+\rho}{2}}|}\sfT_{R,S}\right\}^\kappa  \notag\\[4pt]
&\quad \quad \quad \quad \quad \quad \quad \quad \quad  \times \left(\biint_{Q_{R,S}} \big[1+w^{p} \big]\dxt\right)^\kappa.
\end{align}
\smallskip

\emph{Step 3.} We demonstrate a Moser-type iteration using~\eqref{e.quali-bnd-5-NT}. Fix parameters $0<\sigma^\prime <\sigma \leq 1$ and $i \in \N_0$. Define then  
\[
\begin{cases}
\rho_i:=\sigma^\prime R +\dfrac{\sigma-\sigma^\prime}{2^{i}}R, \quad \tau_i:=\sigma^\prime S +\dfrac{\sigma-\sigma^\prime}{2^{i}}S,\\[8pt]
\widetilde{\rho}_i:=\dfrac{1}{2}(\rho_i+\rho_{i+1}), \quad \widetilde{\tau}_i:=\dfrac{1}{2}(\tau_i+\tau_{i+1}),\\[4pt]
B_i:=B_{\rho_i}, \quad T_i:=(-\tau_i,0], \quad Q_i:=B_i \times T_i, \\[4pt]
\widetilde{B}_i:=B_{\widetilde{\rho}_i}, \quad \widetilde{T}_i:=(-\widetilde{\tau}_i,0], \quad \widetilde{Q}_i:=\widetilde{B}_i \times \widetilde{T}_i.
\end{cases}
\]
Observe, by construction, that
\[
\sigma Q_{R, S}=Q_0 \supset \cdots \supset Q_i \supset Q_{i+1} \supset \cdots \supset Q_\infty=\sigma^\prime Q_{R, S}, \quad \forall i \in \N_0.
\]
Analogously, we define $\{\chi_i\}_{i \in \N_0}$ via $\chi_i:=\kappa^i p$ and then, recursively, for $i \in \N_0$,
\[
\sfW_i:=\left(\biint_{Q_i}\big[1+w^{\chi_i}\big]\dxt \right)^{\nicefrac{1}{\chi_i}} \quad \mbox{and} \quad \sfT_i:=\sfT_{\rho_i, \tau_i}.
\]
Note that $\chi_i \to \infty$ as $i \to \infty$. By an easy exercise, it can be checked that, for every $i \in \N_0$,
\begin{align*}
\frac{|Q_{i}|}{|Q_{i+1}|} \left(\frac{2\rho_i}{\rho_i-\rho_{i+1}}\right)^{d+sp} &\leq 2^{d+1} \left(\frac{2^{i+2}}{\sigma-\sigma^\prime}\right)^{d+sp},\\[4pt]
\left[\frac{\rho_i^{(1-s)p}}{(\rho_i-\rho_{i+1})^p} +\frac{1}{\tau_i-\tau_{i+1}}\right] \frac{|Q_{i}|}{|\widetilde{B}_{i}|} &\leq 2^d\left(\frac{2^{i+1}}{\sigma-\sigma^\prime}\right)^{p}\left(\frac{S}{R^{sp}}+1 \right),\\[4pt]
\frac{\rho_i^{d}}{(\rho_i-\rho_{i+1})^{d+sp}} \frac{|Q_{i}|^{\frac{\eps}{p-1+\eps}}}{|\widetilde{B}_{i}|} \sfT_i&\leq C(d,p,\eps) \frac{4^{i(d+sp)}}{(\sigma-\sigma^\prime)^d}\frac{S}{R^{sp+\frac{d(p-1)}{p-1+\eps}}}\overline{\sfT},
\end{align*}
where
\[
\overline{\sfT}:=\left(\dashint_{-\sigma S}^0 \tail \big(u_+(t)\,; B_{\sigma^\prime R} \big)^{p-1+\eps}\d{t}\right)^{\frac{p-1}{p-1+\eps}}
\]
Thus, appealing to~\eqref{e.quali-bnd-5-NT} over $Q_{i+1}$ and $Q_i$ and rearranging give
\begin{align*}
\biint_{Q_{i+1}}w^{\chi_{i+1}}\dxt &\leq C\Gamma_{p,\eps}^{(p+1)\kappa} \frac{4^{i\left[(d+sp)+(d+p)\kappa \right]}}{(\sigma-\sigma^\prime)^{d+p}}  \\[4pt]
&\quad \times \left[\left(\frac{S}{R^{sp}}+1 \right)+\frac{S}{R^{sp+\frac{d(p-1)}{p-1+\eps}}} \overline{\sfT}\right]^\kappa \left(\biint_{Q_{i}} \left[1+w^{\chi_i} \right]\dxt\right)^\kappa
\end{align*}
and therefore, adding $1$ to the left side implies that the recursive inequality
\begin{equation}\label{e.quali-bnd-6-NT}
\sfW_{i+1}^{\chi_{i+1}} \leq \frac{C}{(\sigma-\sigma^\prime)^{d+p}}\Big[\Gamma_{p,\eps}^{p+1}\sfb^{i} \sfK \sfW_i^{\chi_i}\Big]^\kappa,
\end{equation}
where we have set for notational convenience
\[
\sfb:=4^{(d+sp)+(d+p)}>1 \quad \mbox{and} \quad \sfK:=\left(\frac{S}{R^{sp}}+1 \right)+\frac{S}{R^{sp+\frac{d(p-1)}{p-1+\eps}}}  \overline{\sfT}>1.
\]
Thus, we iterate~\eqref{e.quali-bnd-6-NT} for $i=0,1,\ldots, k$ to obtain that
\begin{align}\label{e.quali-bnd-7-NT}
\left(\biint_{Q_k}w^{\chi_k}\dxt \right)^{\nicefrac{1}{\chi_k}} \leq \sfW_k &\leq \left[\frac{C}{(\sigma-\sigma^\prime)^{d+p}}\right]^{\nicefrac{1}{\chi_k}}\Big[\Gamma_{p,\eps}^{p+1}\sfb^{k-1} \sfK \sfW_{k-1}^{\chi_{k-1}}\Big]^{\nicefrac{\kappa}{\chi_k}} \notag\\
& \,\, \,\vdots \notag\\
&\leq \left[\frac{C}{(\sigma-\sigma^\prime)^{d+p}}\right]^{\mathsf{S}(k)}\Gamma_{p,\eps}^{(p+1)\kappa \mathsf{S}(k)}\sfb^{\mathsf{T}(k)} \sfK^{\kappa \mathsf{S}(k)}\left(\sfW_0^{\chi_0}\right)^{\nicefrac{\kappa^k}{\chi_k}},
\end{align}
where
\[
\mathsf{S}(k):=\sum_{i=0}^{k-1}\frac{\kappa^i}{\chi_k} \quad \mbox{and}\quad \mathsf{T}(k):=\sum_{i=0}^{k-1}i \frac{\kappa^{k-i}}{\chi_k}.
\]
Observe that
\[
\lim_{k \to \infty}\mathsf{S}(k)=\frac{1}{p(\kappa-1)} \quad \mbox{and} \quad \lim_{k \to \infty}\mathsf{T}(k)=\frac{\kappa}{p(\kappa-1)^2}.
\]
Thus, sending $k \to \infty$ in~\eqref{e.quali-bnd-7-NT} yields
\[
\sup_{Q_\infty} w =\limsup_{k \to \infty} \left(\biint_{Q_k}w^{\chi_k}\dxt \right)^{\nicefrac{1}{\chi_k}} \leq \frac{C_{\mathrm{prod}}}{(\sigma-\sigma^\prime)^{\frac{d+p}{p(\kappa-1)}}}\sfK^{\frac{\kappa}{p(\kappa-1)}}\sfW_0
\]
with $C_{\mathrm{prod}}:=C^{\frac{1}{p(\kappa-1)} }\Gamma_{p,\eps}^{\frac{(p+1)\kappa}{p(\kappa-1)}}\sfb^{\frac{\kappa}{p(\kappa-1)^2}}$, namely that, for $\eps \in (1,\infty)$, 
\begin{equation}\label{e.quali-bnd-8-NT}
\sup_{\sigma^\prime Q_{R, S}} u_+ \leq \frac{C_{\mathrm{prod}}^{\frac{p}{p-1+\eps}}}{(\sigma-\sigma^\prime)^{\frac{d+p}{(\kappa-1)(p-1+\eps)}}} \sfK^{\frac{\kappa}{(\kappa-1)(p-1+\eps)}} \left(\biint_{\sigma Q_{R,S}}\big[1+u_+^{p-1+\eps} \big]\dxt \right)^{\nicefrac{1}{(p-1+\eps)}}.
\end{equation}
\smallskip

\emph{Step 4.}\quad In this final step, we verify the estimate~\eqref{e.quali-bnd-8-NT} holds true for all $\eps \in (0,\infty)$. Let $\sigma_1,\sigma_2$ be such that $0<\sigma ^\prime \leq \sigma_1 <\sigma_2 \leq \sigma  <1$. Fixing $\eps \in (1,\infty)$ which satisfies~\eqref{e.quali-bnd-8-NT}, take $\eps^\prime \in (0,1]$ arbitrarily. By~\eqref{e.quali-bnd-8-NT} and Young's inequality we estimate that
\begin{align*}
\sup_{\sigma_1 Q_{R,S}} u_+ &\stackrel{\eqref{e.quali-bnd-8-NT}}{\leq} \frac{C_{\mathrm{prod}}^{\frac{p}{p-1+\eps}}}{(\sigma_2-\sigma_1)^{\frac{d+p}{(\kappa-1)(p-1+\eps)}}} \sfK^{\frac{\kappa}{(\kappa-1)(p-1+\eps)}} \left(\biint_{\sigma_2 Q_{R,S}}\big[1+u_+ \big]^{p-1+\eps}\dxt \right)^{\nicefrac{1}{(p-1+\eps)}}\\
&\leq  
\left[\frac{C_{\mathrm{prod}}^p\sfK^{\frac{\kappa}{\kappa-1}}}{(\sigma_2-\sigma_1)^{\frac{d+p}{\kappa-1}}}\biint_{\sigma_2 Q_{R,S}}\big[1+u_+ \big]^{p-1+\eps^\prime}\dxt \right]^{\nicefrac{1}{(p-1+\eps)}} \left(1+\sup_{\sigma_2 Q_{R,S}}u_+\right)^{\nicefrac{(\eps-\eps^\prime)}{(p-1+\eps)}} \\
&\leq \frac{1}{2}\sup_{\sigma_2 Q_{R,S}}u_+ + C\left[\frac{C_{\mathrm{prod}}^p\sfK^{\frac{\kappa}{\kappa-1}}}{(\sigma_2-\sigma_1)^{\frac{d+p}{\kappa-1}}}\biint_{\sigma_2 Q_{R,S}}\big[1+u_+ \big]^{p-1+\eps^\prime}\dxt \right]^{\nicefrac{1}{(p-1+\eps^\prime)}}.
\end{align*}
The desired statement~\eqref{e.quali-bnd-8-NT} for $\eps \in (0,1]$ follows from using a standard iteration argument (see for example~\cite[Lemma 6.1, Page 191]{Giusti}). The proof is complete.
\end{proof}

We are now ready to prove Theorem~\ref{t.boundedness-NT}. 
\begin{proof}[Proof of Theorem~\ref{t.boundedness-NT}]
The proof is now similar to the one of~\cite[Theorem 2.1]{CN26}. There is a slightly annoying complication caused by the fact that
the auxiliary function $\mathfrak{g}_+(u,k)$ is not estimated by $(u-k)_+^p$ directly, however, this difficulty can be handled quite delicately. 
\smallskip

\emph{Step 1.}  Without loss of generality, we may suppose $z_0=(x_0,t_0)=(0,0)$. Fix $\sigma \in (0,1)$ and let $k \in [0,\infty)$ be arbitrary parameter, to be selected. For $i \in \N_0$, we initialize by setting the following seventeen sequences as introduced in~\cite[Section 3.2]{CN26}, 
\[
\begin{cases}
k_i:=(1-2^{-i})k,\\[4pt]
\rho_i:=\sigma \rho +2^{-i}(1-\sigma)\rho, \quad \theta_i:=\sigma \theta+2^{-i}(1-\sigma)\theta,\\[4pt]
\widetilde{\rho}_i:=\frac{1}{2}(\rho_i+\rho_{i+1}), \quad \widehat{\rho}_i:=\frac{3}{4}\rho_i+\frac{1}{4}\rho_{i+1}, \quad \overline{\rho}_i:=\frac{1}{4}\rho_i+\frac{3}{4}\rho_{i+1}, \\[4pt]
\widetilde{\theta}_i:=\frac{1}{2}(\theta_i+\theta_{i+1}), \quad \,\widehat{\theta}_i:=\frac{3}{4}\theta_i+\frac{1}{4}\theta_{i+1}, \quad \,\overline{\theta}_i:=\frac{1}{4}\theta_i+\frac{3}{4}\theta_{i+1}, \\[4pt]
B_i:=B_{\rho_i}, \quad \widetilde{B}_i:=B_{\widetilde{\rho}_i}, \quad \widehat{B}_i:=B_{\widehat{\rho}_i}, \quad \overline{B}_i:=B_{\overline{\rho}_i}, \\[4pt]
Q_i:=B_i \times (-\theta_i,0], \quad \widetilde{Q}_i:=\widetilde{B}_i \times (-\widetilde{\theta}_i,0], \\[4pt]
\widehat{Q}_i:=\widehat{B}_i \times (-\widehat{\theta}_i, 0], \quad \overline{Q}_i:=\overline{B}_i \times (-\overline{\theta}_i, 0]. 
\end{cases}
\]
We take a smooth cutoff function $\zeta : \R^{d+1}  \to [0,1]$ satisfying
\[
\1_{\widetilde{Q}_i} \leq \zeta \leq \1_{\widehat{Q}_i}, \quad |\nabla \zeta| \leq \frac{2^{i+4}}{(1-\sigma)\rho} \quad \mbox{and} \quad |\partial_t\zeta| \leq \frac{2^{i+4}}{(1-\sigma)\theta}.
\]
For $\eps \in (0,\infty)$, assuming $u \in L^{p-1+\eps}_{\mathrm{loc}}(\Omega_T)$ temporarily, set $\mathsf{M}_0:=\sup_{Q_0} u_+$. Note that $\mathsf{M}_0$ is finite by the qualitative local boundedness (Proposition~\ref{t.qualitative-bnd-NT}). We then prepare families of the levels, for $i \in \N_0$, that
\begin{align*}
\mathsf{K}_i&:=\left[\left(\ell(t)+\frac{1}{2^{\frac{(2-p)_+}{p-1}}4}\mathsf{M}_0\right)^{p-1}+k_{i}^{p-1}\right]^{\nicefrac{1}{(p-1)}} \quad \mbox{and} \\[4pt]
\ell(t)&:=\left[\boldsymbol{\Gamma}\int_{-\theta}^t \int_{\R^d \setminus B_R}\frac{u_+^{p-1}(y,\tau)}{|y|^{d+sp}}\dytau + (2-p)_+ \left(\frac{\mathsf{M}_0}{4^{\frac{(2-p)_+}{p-1}}8} \right)^{p-1}\right]^{\nicefrac{1}{(p-1)}}
\end{align*}
with
\[
k_i:=(1-2^{-i})k \quad \mbox{and} \quad \boldsymbol{\Gamma}:=2^{\frac{(2-p)_+^2}{p-1}} 8^{(2-p)_+}(p-1) \left(\frac{R}{R-\rho}\right)^{d+sp}\frac{1}{\chi_p},
\]
where 
\begin{equation}\label{e.bnd-est-0-NT}
\chi_p:=
\begin{cases}
1 \quad &\mbox{if} \quad p \geq 2,\\
\left[1+\frac{2^{\frac{2-p}{p-1}+1}}{(2-p)^{\frac{1}{p-1}}}\right]^{p-2} \quad &\mbox{if} \quad 1<p<2.
\end{cases}
\end{equation}
Here, we temporary enforce that
\[
k \geq \ell(0)+\frac{1}{2^{\frac{(2-p)_+}{p-1}} 4}\mathsf{M}_0 \quad \Longrightarrow \quad k \geq \ell(t)+\frac{1}{2^{\frac{(2-p)_+}{p-1}} 4}\mathsf{M}_0, \quad \forall t \in (-\theta, 0]. 
\]
By construction, it is also apparent that
\begin{align*}
  \left\{
    \begin{array}{l}
      \mathsf{K}_0= \ell(t)+\frac{1}{2^{\frac{(2-p)_+}{p-1}} 4}\mathsf{M}_0\\[4pt]
\mathsf{K}_i < \widetilde{\mathsf{K}}_i <\mathsf{K}_{i+1} \,\uparrow \,\mathsf{K}_\infty=\left[\left(\ell(t)+\dfrac{1}{2^{\frac{(2-p)_+}{p-1}} 4}\mathsf{M}_0\right)^{p-1}+k^{p-1}\right]^{\nicefrac{1}{(p-1)}}, \quad \forall i \in \N_0.
    \end{array}
    \right.
\end{align*}
In the Caccioppoli estimate~\eqref{e.caccioppoli1}, we then select the time-dependent level $k(t)$ so that
\[
k(t)=\widetilde{\mathsf{K}}_i:=\left[\left(\ell(t)+\frac{1}{2^{\frac{(2-p)_+}{p-1}} 4}\mathsf{M}_0\right)^{p-1}+\widetilde{k}_{i}^{p-1}\right]^{\nicefrac{1}{(p-1)}}
\]
and, to keep the notation short, denote
\[
\widetilde{w}_i:= (u-\widetilde{\mathsf{K}}_i)_+ \quad \mbox{and} \quad w_{i}:=(u-\mathsf{K}
_i)_+.
\]
Under these choice of parameters, we obtain from~\eqref{e.caccioppoli1}$_+$ that
\begin{align}\label{e.bnd-est-1-NT}
&\underbrace{\sup_{t \in (-\widetilde{\theta}_i,0]}\int_{\widetilde{B}_i}\mathfrak{g}_+(u,k(t))\d{x}}_{=:\sfS}+\int_{-\widetilde{\theta}_i}^0 \iint_{\widetilde{B}_i \times \widetilde{B}_i} \frac{|\widetilde{w}_i(x,t)-\widetilde{w}_i(y,t)|^p}{|x-y|^{d+sp}}\dxyt 
\notag\\[4pt]
& \quad \quad \leq C\iint_{Q_i}\left|\partial_t \zeta^p\right|\mathfrak{g}_+(u,k(t))\dxt \notag\\[4pt]
&\quad \quad \quad +C\int_{-\theta_i}^{0}\iint_{B_i \times B_i} \min\left\{\widetilde{w}_i^p(x,t), \widetilde{w}_i^p(y,t)\right\}\dfrac{\big|\zeta(x,t)-\zeta(y,t)\big|^p}{|x-y|^{d+sp}}\dxyt \notag\\[4pt]
&\quad \quad \quad +C\iint_{Q_i}\zeta^p\widetilde{w}_i(x,t)\d{x} \left(\sup_{x\,\in\, \widehat{B}_i}\int_{\R^d \setminus B_i}\frac{\widetilde{w}_i^{p-1}(y,t)}{|x-y|^{d+sp}}\d{y}\right)\d{t} \notag\\[4pt]
&\quad \quad  \quad -C\iint_{Q_i}\zeta^pk^\prime(t)\Big(|u|+|k(t)|\Big)^{p-2}\widetilde{w}_i(x,t)\dxt\notag\\[4pt]
&\quad \quad =:  \sfI+\sfI\sfI+\sfI\sfI\sfI+\sfI\sfV
\end{align}
whenever concentric cylinders $\widetilde{Q}_i$ and $Q_i$, where the definitions of $\sfI$--$\sfI\sfV$ are clear from the context. Using Lemma~\ref{t.g}, we crudely bound
\begin{align}\label{e.bnd-est-2-NT}
\sfI &\leq C\frac{2^i}{(1-\sigma)\theta}\iint_{Q_i}\left(|u|+|\widetilde{\mathsf{K}}_i|\right)^{p-2}\left(u-\widetilde{\mathsf{K}}_i\right)_+^2\dxt\notag \\
&\leq C\frac{2^i}{(1-\sigma)\theta}\iint_{Q_i}\left(|u|+|\widetilde{\mathsf{K}}_i|\right)^{p-1}\left(u-\widetilde{\mathsf{K}}_i\right)_+\dxt.
\end{align}
Set $\sfG_p:=\widetilde{\mathsf{K}}_i-\mathsf{K}_i$. By an easy exercise, it can be checked that
\[
\sfG_p \geq \frac{1}{c(p)} \frac{k}{2^{i(2 \vee p)}},
\]
therefore the condition $k \geq \ell(t)+\frac{1}{2^{\frac{(2-p)_+}{p-1}} 4}\mathsf{M}_0$ for all $t$ yields that
\begin{align}\label{e.bnd-est-3-NT}
\widetilde{\mathsf{K}}_i-\mathsf{K}_i=\sfG_p &\geq \frac{1}{c(p)\,2^{i(2 \vee p)}} \left(\frac{k^{p-1}+(\ell(t)+\nicefrac{\mathsf{M}_0}{(2^{\frac{(2-p)_+}{p-1}} 4)})^{p-1}}{2}\right)^{\nicefrac{1}{(p-1)}} \notag\\
&=\frac{\mathsf{K}_\infty}{C(p)\,2^{i(2 \vee p)}}.
\end{align}
Since $0 \leq \mathsf{K}_i < \widetilde{\mathsf{K}}_i \leq \mathsf{K}_\infty$ we observe that, for $u > \widetilde{\mathsf{K}}_i$
\begin{align*}
\left(|u|+|\widetilde{\mathsf{K}}_i|\right)^{p-1}\left(u-\widetilde{\mathsf{K}}_i\right)_+ &\leq 2^{p-1} \left(|u|^p+|\widetilde{\mathsf{K}}_i|^p\right)  \\
&\leq 4^p\left[ (u-\widetilde{\mathsf{K}}_i)^p+\mathsf{K}_\infty^p \right]\\
&\!\!\!\stackrel{\eqref{e.bnd-est-3-NT}}{\leq} C(p)2^{ip(2 \vee p)}\left[ (u-\mathsf{K}_i)^p+(\widetilde{\mathsf{K}}_i-\mathsf{K}_i)^p \right]\\
&\leq C(p)2^{ip(2 \vee p)}(u-\mathsf{K}_i)^p
\end{align*}
for a constant $C(p)<\infty$. Thus, by~\eqref{e.bnd-est-3-NT} a straightforward computation shows that
\begin{equation}\label{e.bnd-est-a-NT}
\sfI \leq C(p)\frac{2^{i[p(2 \vee p)+1]}}{(1-\sigma)\theta} \iint_{Q_i}w_i^p\dxt.
\end{equation}
An easy computation show that, 
\begin{equation}\label{e.bnd-est-b-NT}
\sfI\sfI \leq \frac{c2^{ip}}{(1-\sigma)^p \rho^{sp}}\iint_{Q_i} w_i^p\dxt.
\end{equation}
On the other hand, similarly to the proof of~\cite[Section 3.2]{CN26}, we argue that
\begin{align}\label{e.bnd-est-c-NT}
\sfI\sfI\sfI+\sfI\sfV &\leq C 
\frac{2^{i(d+sp)}}{[\sigma (1-\sigma)]^{d+sp}\rho^{sp}} \left(\frac{\boldsymbol{\mathfrak{S}}}{k}\right)^{p-1}\iint_{Q_i}w_i^p\dxt,
\end{align}
under the restriction
\[
k \geq \left[\boldsymbol{\Gamma}\int_{-\theta}^0 \int_{\R^d \setminus B_R}\frac{u_+^{p-1}(y,\tau)}{|y|^{d+sp}}\dytau + (2-p)_+ \left(\frac{\mathsf{M}_0}{4^{\frac{(2-p)_+}{p-1}}8} \right)^{p-1}\right]^{\nicefrac{1}{(p-1)}}=\ell(0),
\]
where \[
\displaystyle \boldsymbol{\mathfrak{S}}:=\left[\left(\frac{R}{\rho}\right)^d \sup_{t \in (-\theta,0]}\dashint_{B_R}u_+^{p-1}(y,t)\d{y}\right]^{\nicefrac{1}{(p-1)}}.
\]
Indeed, we begin with splitting $\sfI\sfI\sfI$ into two pieces:
\begin{align*}
\sfI\sfI\sfI_1&:=C\iint_{Q_i}\zeta^p\widetilde{w}_i(x,t)\d{x} \left(\sup_{x\,\in\, \widehat{B}_i}\int_{B_R \setminus B_i}\frac{\widetilde{w}_i^{p-1}(y,t)}{|x-y|^{d+sp}}\d{y}\right)\d{t}, \\
\sfI\sfI\sfI_2&:=C\iint_{Q_i}\zeta^p\widetilde{w}_i(x,t)\d{x} \left(\sup_{x\,\in\, \widehat{B}_i}\int_{\R^d \setminus B_R}\frac{\widetilde{w}_i^{p-1}(y,t)}{|x-y|^{d+sp}}\d{y}\right)\d{t}.
\end{align*}
Observe that for any $x \in \widehat{B}_i$ and $y \in B_R \setminus B_i$,
\[
\frac{|y-x|}{|y|} \geq 1-\frac{|x|}{|y|} \geq \frac{1-\sigma}{2^{i+3}},
\]
while, if $y \in \R^d \setminus B_R$ and $x \in \widehat{B}_i$ then
\[
\frac{|y-x|}{|y|} \geq 1-\frac{|x|}{|y|} \geq \frac{R-\rho}{R}.
\]
Thus, we arrive at
\begin{align}\label{e.bnd-est-d-NT}
\sfI\sfI\sfI_1 +\sfI\sfI\sfI_2&\leq C\iint_{Q_i}\zeta^p\widetilde{w}_i(x,t)\d{x} \left[\left(\frac{2^i}{1-\sigma}\right)^{d+sp}\int_{B_R \setminus B_i}\frac{\widetilde{w}_i^{p-1}(y,t)}{|y|^{d+sp}}\d{y}\right]\d{t} \notag\\
& \quad +C\iint_{Q_i}\zeta^p\widetilde{w}_i(x,t)\d{x} \left[\left(\frac{R}{R-\rho}\right)^{d+sp}\int_{\R^d \setminus B_R}\frac{\widetilde{w}_i^{p-1}(y,t)}{|y|^{d+sp}}\d{y}\right]\d{t}.
\end{align}
Next, we turn our attention to $\sfI\sfV$. For this we split the estimate of $\sfI\sfV$ into two cases $p \geq 2$ and $1<p<2$. In the case $p \geq 2$, we crudely estimate that
\begin{align*}
\sfI\sfV&= -C\iint_{Q_i}\zeta^pk^\prime(t)\Big(|u|+|k(t)|\Big)^{p-2}\widetilde{w}_i(x,t)\dxt\\
&\leq  -C\iint_{Q_i}\zeta^pk^\prime(t)k(t)^{p-2}\widetilde{w}_i(x,t)\dxt \\
&=-\frac{C}{p-1}\iint_{Q_i}\zeta^p\frac{\d{}}{\d{t}}k(t)^{p-1}\widetilde{w}_i(x,t)\dxt.
\end{align*}
In the remaining case $1<p<2$, by observation, for $u >k(t)$, that
\[
|u|+|k(t)| <2u \leq 2\mathsf{M}_0 \quad \Longrightarrow \quad -\frac{1}{\Big(|u|+|k(t)|\Big)^{2-p}} \leq -\frac{1}{(2\mathsf{M}_0)^{2-p}}.
\]
This, together with the nonnegativity of $k^\prime(t)$ and 
\[
k(t)=\widetilde{\mathsf{K}}_i \geq \frac{1}{2^{\frac{(2-p)_+}{p-1}}4}\mathsf{M}_0
\]
yields that
\begin{align*}
\sfI\sfV&=-\frac{C}{p-1}\iint_{Q_i}\zeta^p \frac{\d{}}{\d{t}}k^\prime(t)^{p-1} \frac{k(t)^{2-p}}{\Big(|u|+|k(t)|\Big)^{2-p}}\widetilde{w}_i(x,t)\dxt\\
&\leq -\frac{C}{p-1}\iint_{Q_i}\zeta^p \frac{\d{}}{\d{t}}k^\prime(t)^{p-1} \frac{1}{(2\mathsf{M}_0)^{2-p}} \left(\frac{1}{2^{\frac{(2-p)_+}{p-1}}4}\mathsf{M}_0\right)^{2-p}\widetilde{w}_i(x,t)\dxt\\
&=-\frac{C}{(p-1)8^{2-p}2^{\frac{(2-p)^2}{p-1}}}\iint_{Q_i}\zeta^p\frac{\d{}}{\d{t}}k(t)^{p-1}\widetilde{w}_i(x,t)\dxt.
\end{align*}
In conclusion, we have estimated that
\[
\sfI\sfV \leq -\frac{C}{(p-1)8^{(2-p)_+}2^{\frac{(2-p)_+^2}{p-1}}}\iint_{Q_i}\zeta^p\frac{\d{}}{\d{t}}k(t)^{p-1}\widetilde{w}_i(x,t)\dxt.
\]
To estimate further, we need to distinguish between the cases $p \geq 2$ and $1<p<2$. When dealing with $p \geq 2$, it is straightforward to check that
\[
\frac{\d{}}{\d{t}}k(t)^{p-1}=(p-1) \left(\ell(t)+\frac{1}{4}\mathsf{M}_0\right)^{p-2} \frac{\d{}}{\d{t}}\ell(t) \geq \frac{\d{}}{\d{t}}\ell(t)^{p-1}.
\]
In the latter case $1<p<2$, the definition of $\ell(t)$ implies that
\[
\ell(t)+\frac{1}{2^{\frac{2-p}{p-1}}4}\mathsf{M}_0 \leq \left[1+\frac{2^{\frac{2-p}{p-1}+1}}{(2-p)^{\frac{1}{p-1}}}\right]\ell(t),
\]
namely
\[
\left(\ell(t)+\frac{1}{2^{\frac{2-p}{p-1}}4}\mathsf{M}_0 \right)^{p-2} \geq \left[1+\frac{2^{\frac{2-p}{p-1}+1}}{(2-p)^{\frac{1}{p-1}}}\right]^{p-2}\ell(t)^{p-2},
\]
thereby getting
\begin{align*}
\frac{\d{}}{\d{t}}k(t)^{p-1}&=(p-1)\left(\ell(t)+\frac{1}{2^{\frac{2-p}{p-1}}4}\mathsf{M}_0 \right)^{p-2} \frac{\d{}}{\d{t}}\ell(t)\\
&\geq \left[1+\frac{2^{\frac{2-p}{p-1}+1}}{(2-p)^{\frac{1}{p-1}}}\right]^{p-2}\frac{\d{}}{\d{t}}\ell(t)^{p-1}.
\end{align*}
Altogether, we have shown that
\[
\sfI \sfV \leq -\frac{C}{(p-1)8^{(2-p)_+}2^{\frac{(2-p)_+^2}{p-1}}}\iint_{Q_i}\chi_p\zeta^p\frac{\d{}}{\d{t}}\ell(t)^{p-1}\widetilde{w}_i(x,t)\dxt
\]
where the constant $\chi_p$ is defined by~\eqref{e.bnd-est-0-NT}. Consequently, using
\[
\frac{\d{}}{\d{t}}\ell(t)^{p-1}=\boldsymbol{\Gamma} \int_{\R^d \setminus B_R} \frac{u_+^{p-1}(y,t)}{|y|^{d+sp}}\d{y}
\]
implies $\sfI\sfI\sfI_2+\sfI\sfV \leq 0$, and therefore
\begin{align*}
\sfI\sfI\sfI&+\sfI\sfV \leq \sfI\sfI\sfI_1+\sfI\sfI\sfI_2+\sfI\sfV\\
&\leq C\iint_{Q_i}\zeta^p\widetilde{w}_i(x,t)\d{x} \left[\left(\frac{2^i}{1-\sigma}\right)^{d+sp}\int_{B_R \setminus B_i}\frac{\widetilde{w}_i^{p-1}(y,t)}{|y|^{d+sp}}\d{y}\right]\d{t}\\
&\leq C\left(\frac{2^i}{1-\sigma}\right)^{d+sp}\frac{R^d}{(\sigma \rho)^{d+sp}} \iint_{Q_i}\widetilde{w}_i \d{x}\left[\dashint_{B_R} w_i^{p-1}(y,t)\d{y}\right]\d{t}\\
&\leq C\frac{2^{i(d+sp)}}{\left[(1-\sigma)\sigma\right]^{d+sp}\rho^{sp}} \left(\frac{R}{\rho}\right)^d \left[\sup_{t \in (-\theta, 0]}\dashint_{B_R \times \{t\}}w_i^{p-1}\d{y}\right]\times \iint_{Q_i}\widetilde{w}_i\dxt
\end{align*}
for a constant $C(\data)<\infty$. This is~\eqref{e.bnd-est-c-NT}.

Combining the previous three estimates~\eqref{e.bnd-est-a-NT}--\eqref{e.bnd-est-c-NT}, we have 
\begin{align}\label{e.bnd-est-4-NT}
&\sfS+\int_{-\widetilde{\theta}_i}^0 \iint_{\widetilde{B}_i \times \widetilde{B}_i}  \frac{|\widetilde{w}_i(x,t)-\widetilde{w}_i(y,t)|^p}{|x-y|^{d+sp}}\dxyt 
\notag\\[4pt]
&\quad \quad \leq  \sfI+\sfI\sfI+\sfI\sfI\sfI+\sfI\sfV \notag\\[4pt]
& \quad \quad \leq C \frac{\sfb^i}{(1-\sigma)^{d+p}} \left[\frac{1}{\rho^{sp}}+\frac{1}{\theta}+\frac{1}{\sigma^{d+sp} \rho^{sp}}\left(\frac{\boldsymbol{\mathfrak{S}}}{k}\right)^{p-1}\right]\iint_{Q_i}w_i^p\dxt \notag\\
& \quad \quad \leq C \frac{\sfb^i}{(1-\sigma)^{d+p}\rho^{sp}} \left[\frac{\rho^{sp}}{\theta}+\frac{1}{\delta^{p-1}\sigma^{d+sp}}\right]\iint_{Q_i}w_i^p\dxt
\end{align}
for a constant $C(\data)<\infty$ and $\sfb:=2^{d+sp+p(2 \vee p)+p}>1$. To obtain the last line we have imposed 
\begin{equation}\label{e.k-condition-NT}
k \geq \delta \boldsymbol{\mathfrak{S}} \vee \left(\ell(0)+\frac{1}{2^{\frac{(2-p)_+}{p-1}} 4}\mathsf{M}_0\right)
\end{equation}
where $\delta \in (0,1]$ is to be selected below.
\smallskip

\emph{Step 2.} To proceed further, we split the observation into two cases: $p\geq 2$ and $ p \in (1,2)$. In the case $p \geq 2$, using Lemma~\ref{t.g} and~\eqref{e.bnd-est-4-NT}, we crudely estimate that, for every $i \in \N_0$, 
\begin{align*}
\sup_{t \in (-\widetilde{\theta}_i,0]}\dashint_{\widetilde{B}_i \times \{t\}}\tilde{w}_{i}^p\d{x} &\leq \sup_{t \in (-\widetilde{\theta}_i,0]}\dashint_{\widetilde{B}_i \times \{t\}} \left(|u|+|\widetilde{\mathsf{K}}_i|\right)^{p-2} \left(u-\widetilde{\mathsf{K}}_i\right)_+^2\d{x}\\
&\leq \frac{1}{c_1(p)|\widetilde{B}_i|}\sup_{t \in (-\widetilde{\theta}_i,0]}\int_{\widetilde{B}_i \times \{t\}} \mathfrak{g}_+(u,\widetilde{\mathsf{K}}_i)\d{x}\\
& \!\!\!\stackrel{\eqref{e.bnd-est-4-NT}}{\leq} \frac{C}{|\widetilde{B}_i|}\left[\frac{\sfb^i}{(1-\sigma)^{d+p}\rho^{sp}} \left(\frac{\rho^{sp}}{\theta}+\frac{1}{\delta^{p-1}\sigma^{d+sp}}\right)\iint_{Q_i}w_i^p\dxt\right].
\end{align*}

We next consider the case $1< p <2$. Using the triangle inequality and H\"{o}lder's inequality with $\left(\nicefrac{2}{p}, \nicefrac{2}{(2-p)}\right)$ and, subsequently, appealing to Lemma~\ref{t.g} and~\eqref{e.bnd-est-4-NT} yields, under~\eqref{e.k-condition-NT}, that
\begin{align*}
\sup_{t \in (-\widetilde{\theta}_i,0]}&\dashint_{\widetilde{B}_i \times \{t\}}\tilde{w}_{i}^p\d{x} \\[4pt]
&=\sup_{t \in (-\widetilde{\theta}_i,0]}\dashint_{\widetilde{B}_i \times \{t\}} \left(|u|+|\widetilde{\mathsf{K}}_i|\right)^{\nicefrac{p(p-2)}{2}}\left(u-\widetilde{\mathsf{K}}_i\right)_+^p \cdot \left(|u|+|\widetilde{\mathsf{K}}_i|\right)^{\nicefrac{p(2-p)}{2}}\d{x} \\[4pt]
&\leq \left(\sup_{t \in (-\widetilde{\theta}_i,0] }\dashint_{\widetilde{B}_i \times \{t\}}\left(|u|+|\widetilde{\mathsf{K}}_i|\right)^{p-2}\left(u-\widetilde{\mathsf{K}}_i\right)_+^2\d{x} \right)^{\nicefrac{p}{2}} \\
&\quad \quad \quad \quad \times \left(\sup_{t \in (-\widetilde{\theta}_i,0] } \dashint_{\widetilde{B}_i \times \{t\}}\left(|u|+|\widetilde{\mathsf{K}}_i|\right)^p \1_{\{u(\cdot,t)>\widetilde{\mathsf{K}}_i\}}\d{x} \right)^{\nicefrac{(2-p)}{2}}\\[4pt]
&\leq \frac{2^{\nicefrac{p(2-p)}{2}}}{c_1^{\nicefrac{p}{2}}}\left(\sup_{t \in (-\widetilde{\theta}_i,0] }\dashint_{\widetilde{B}_i \times \{t\}}\mathfrak{g}_+(u,\widetilde{\mathsf{K}}_i)\d{x} \right)^{\nicefrac{p}{2}}\mathsf{M}_0^{\nicefrac{p(2-p)}{2}}\\[4pt]
&\leq C(p)\left(\frac{1}{|\widetilde{B}_i|}\sup_{t \in (-\widetilde{\theta}_i,0] }\int_{\widetilde{B}_i \times \{t\}}\mathfrak{g}_+(u,\widetilde{\mathsf{K}}_i)\d{x}\right)^{\nicefrac{p}{2}} \left(\frac{\mathsf{M}_0}{2^{\frac{(2-p)}{p-1}}4}\right)^{\nicefrac{p(2-p)}{2}}\\[4pt]
&\!\!\!\stackrel{\eqref{e.bnd-est-4-NT}}{\leq} C\frac{k^{\nicefrac{p(2-p)}{2}}}{|\widetilde{B}_i|^{\nicefrac{p}{2}}}\left[\frac{\sfb^i}{(1-\sigma)^{d+p}\rho^{sp}} \left(\frac{\rho^{sp}}{\theta}+\frac{1}{\delta^{p-1}\sigma^{d+sp}}\right)\iint_{Q_i}w_i^p\dxt\right]^{\nicefrac{p}{2}}
\end{align*}
with $C(\data)<\infty$ in the case $1<p<2$.

Consequently, in any case $p \in (1,\infty)$, we have proven that for $k$ large enough to obey~\eqref{e.k-condition-NT},
\begin{align}\label{e.bnd-est-6-NT}
\sup_{t \in (-\widetilde{\theta}_i,0]}&\,\dashint_{\widetilde{B}_i \times \{t\}}\tilde{w}_{i}^p\d{x} \notag\\
&\leq C\frac{k^{\frac{p(2-p)_+}{2}}}{|\widetilde{B}_i|^{\frac{p \wedge 2}{2}}}\left[\frac{\sfb^i}{(1-\sigma)^{d+p}\rho^{sp}} \left(\frac{\rho^{sp}}{\theta}+\frac{1}{\delta^{p-1}\sigma^{d+sp}}\right)\iint_{Q_i}w_i^p\dxt\right]^{\frac{p \wedge 2}{2}},
\end{align}
where the constant $C<\infty$ depends only on $\data$.
\smallskip

\emph{Step 3.} We argue next that there exists $C(\data)<\infty$ such that, for every $i \in \N_0$, 
\begin{align}\label{e.bnd-est-7-NT}
\iint_{Q_{i+1}}&w_{i+1}^p\dxt \notag\\
&\leq C\left[\frac{\sfb^i}{(1-\sigma)^{d+p}} \left(\frac{\rho^{sp}}{\theta}+\frac{1}{\delta^{p-1}\sigma^{d+sp}}\right)\right]^{\nicefrac{1}{\kappa}}\left[\iint_{Q_i}w_i^p\dxt\right] k^{-p(1-\frac{1}{\kappa})}\notag\\
&\quad \quad \quad \quad \quad \times \left[\sup_{t \in (-\widetilde{\theta}_i,0]}\biint_{\widetilde{B}_i \times \{t\}}\tilde{w}_{i}^p\d{x}\right]^{\frac{\kappa_\ast-1}{\kappa \kappa_\ast}}.
\end{align}
Indeed, let $\varphi: \R^{d+1} \to [0,1]$ be a smooth cutoff function satisfying
\[
\1_{Q_{i+1}} \leq \varphi \leq \1_{\widetilde{Q}_i}, \quad |\nabla \varphi| \leq 2^{i+4}/\rho, \quad \mathrm{supp}\,(\varphi) \subset \widetilde{Q}_i.
\]
By the combination of~Proposition~\ref{FS} with $\mathsf{d}=2^{-i-4}$ and H\"{o}lder's inequality implies that
\begin{align}\label{e.bnd-est-8-NT}
\iint_{Q_{i+1}}w_{i+1}^p\dxt &\leq C \Bigg[\rho^{sp} \int_{-\widetilde{\theta}_i}^0\iint_{\widetilde{B}_i \times \widetilde{B}_i}  \frac{\left|\varphi w_{i+1}(x,t)-\varphi w_{i+1}(y,t)\right|^p}{|x-y|^{d+sp}}\dxyt  \notag\\
&\quad \quad \quad \quad \quad \quad \quad \quad \quad \quad +\frac{1}{\mathsf{d}^{d+sp}}\iint_{\widetilde{Q}_i}(\varphi w_{i+1})^p\dxt\Bigg]^{\nicefrac{1}{\kappa}} \notag\\
&\quad \quad \quad \quad \quad \times \left[\sup_{t \in (-\widetilde{\theta}_i,0]}\dashint_{\widetilde{B}_i \times \{t\}}(\varphi w_{i+1})^p\d{x}\right]^{\frac{\kappa_\ast-1}{\kappa \kappa_\ast}}|\widetilde{A}_i|^{1-\nicefrac{1}{\kappa}},
\end{align}
where $\kappa=1+\frac{\kappa_\ast-1}{\kappa_\ast}$ with $\kappa_\ast$ as in Proposition~\ref{FS} and 
\[
\widetilde{A}_i:=\left\{u(x,t)>\widetilde{K}_i \right\} \cap Q_i.
\]
Then, $|\widetilde{A}_i|$ has the following Chebyshev type estimate, which is easy to check: for any $i \in \N_0$
\begin{equation}\label{e.bnd-est-9-NT}
|\widetilde{A}_i| \leq c(p)\frac{2^{(i+2)p(2 \vee p)}}{k^p}\iint_{Q_i}w_i^p\dxt.
\end{equation}
This, together with Lemma~\ref{t.alg-est-3} and~\eqref{e.bnd-est-4-NT}, yields that
\begin{align}\label{e.bnd-est-10-NT}
&\rho^{sp}\int_{-\widetilde{\theta}_i}^0\iint_{\widetilde{B}_i \times \widetilde{B}_i}  \frac{\left|\varphi w_{i+1}(x,t)-\varphi w_{i+1}(y,t)\right|^p}{|x-y|^{d+sp}}\dxyt \notag\\
& \quad \leq C \frac{\sfb^i}{(1-\sigma)^{d+p}} \left(\frac{\rho^{sp}}{\theta}+\frac{1}{\delta^{p-1}\sigma^{d+sp}}\right)\iint_{Q_i}w_i^p\dxt
\end{align}
for a constant $C(\data)<\infty$. Thus, the statement~\eqref{e.bnd-est-7-NT} follows from the three displays~\eqref{e.bnd-est-8-NT}--\eqref{e.bnd-est-10-NT}. 

To shorten the notation, denote
\[
\sfY_{i}:=\iint_{Q_i}w_i^p\dxt \quad \mbox{and} \quad \varkappa_p:=\frac{1}{\kappa}+\frac{p \wedge 2}{2}\frac{\kappa_\ast-1}{\kappa \kappa_\ast}.
\]
In the light of~\eqref{e.bnd-est-6-NT} and~\eqref{e.bnd-est-7-NT} we obtain, for every $i \in \N_0$, that
\begin{align*}
\sfY_{i+1} &\leq \frac{C \widetilde{\sfb}^i \left(\frac{\rho^{sp}}{\theta}+\frac{1}{\delta^{p-1}\sigma^{d+sp}}\right)^{\frac{1}{\kappa}+\frac{p \wedge 2}{2}\frac{\kappa_\ast-1}{\kappa \kappa_\ast}}}{(1-\sigma)^{(d+p)\left(\frac{1}{\kappa}+\frac{p \wedge 2}{2}\frac{\kappa_\ast-1}{\kappa \kappa_\ast}\right)}}\cdot \frac{k^{\frac{p(2-p)_+}{2}\cdot\frac{\kappa_\ast-1}{\kappa \kappa_\ast}-p(1-\frac{1}{\kappa})} }{\rho^{(d+sp)\frac{p \wedge 2}{2}\frac{\kappa_\ast-1}{\kappa \kappa_\ast}}}\sfY_{i} ^{1+\frac{p \wedge 2}{2}\frac{\kappa_\ast-1}{\kappa \kappa_\ast}}\\
&= \frac{C\widetilde{\sfb}^i \left(\frac{\rho^{sp}}{\theta}+\frac{1}{\delta^{p-1}\sigma^{d+sp}}\right)^{\varkappa_p}}{(1-\sigma)^{(d+p)\varkappa_p}}\frac{\sfY_{i} ^{1+\frac{p \wedge 2}{2}\frac{\kappa_\ast-1}{\kappa \kappa_\ast}}}{(\rho^{d+sp}k^p)^{\varkappa_p-\nicefrac{1}{\kappa}}}
\end{align*}
with $C(\data)<\infty$ and $\widetilde{\sfb}:=2^{d\frac{\kappa_\ast-1}{\kappa \kappa_\ast}}\sfb>1$, because
\[
\frac{p(2-p)_+}{2}\cdot\frac{\kappa_\ast-1}{\kappa \kappa_\ast}-p\left(1-\frac{1}{\kappa}\right)=-p\frac{p \wedge 2}{2}\frac{\kappa_\ast-1}{\kappa \kappa_\ast}=-p\left(\varkappa_p-\frac{1}{\kappa}\right).
\]
Now, set
\begin{align*}
&q_\ast:=\frac{2}{p(p \wedge 2)}\varkappa_p \frac{\kappa \kappa_\ast}{\kappa_\ast-1} \quad \mbox{and} \\
&\overline{\boldsymbol{\cA}}_\delta \equiv \overline{\boldsymbol{\cA}}_\delta\left(\sigma,\rho,R,\theta\right):=\left(\frac{\rho^{sp}}{\theta}+\frac{1}{\delta^{p-1}\sigma^{d+sp}}\right)^{q_\ast}\left(\frac{\theta}{\rho^{sp}}\right)^{\nicefrac{1}{p}}.
\end{align*}
If we impose another restriction on $k$, namely that
\begin{align}\label{e.k-condition3-NT}
&\sfY_{0}\leq \left[ \frac{C \left(\frac{\rho^{sp}}{\theta}+\frac{1}{\delta^{p-1}\sigma^{d+sp}}\right)^{\varkappa_p}}{(1-\sigma)^{(d+p)\varkappa_p}(\rho^{d+sp}k^p)^{\varkappa_p-1/\kappa}}\right]^{-\nicefrac{1}{\left(\frac{p \wedge 2}{2}\frac{\kappa_\ast-1}{\kappa\kappa_\ast}\right)}}\widetilde{\sfb}^{-\nicefrac{1}{\left(\frac{p \wedge 2}{2}\frac{\kappa_\ast-1}{\kappa\kappa_\ast}\right)^2}} \notag \\[4pt]
\iff & \quad k \geq \frac{\overline{C}}{(1-\sigma)^{(d+p)q_\ast}}\overline{\boldsymbol{\cA}}_\delta\left(\biint_{Q_{\rho,\theta}}\left[u-\left(\ell(t)+\frac{1}{2^{\frac{(2-p)_+}{p-1}} 4}\mathsf{M}_0\right)\right]_+^p\dxt \right)^{\nicefrac{1}{p}}  
\end{align}
for a constant $\overline{C}(\data) <\infty$, where to obtain the last line $\widetilde{\sfb}(d,s,p)>1$ was absorbed into a lump constant $\overline{C}$. Thus, Lemma~\ref{t.FGC} says that
\begin{equation}\label{e.bnd-est-11-NT}
\sfY_{i} \to 0 \quad \mbox{as} \quad i \to \infty. 
\end{equation}
Select $k \in [0,\infty)$ such that
\[
k=\left(\ell(0)+\frac{1}{2^{\frac{(2-p)_+}{p-1}}4}\mathsf{M}_0\right)+\delta \boldsymbol{\mathfrak{S}}+\frac{\overline{C}}{(1-\sigma)^{(d+p)q_\ast}}\overline{\boldsymbol{\cA}}_\delta\left(\biint_{Q_{\rho,\theta}}u_+^p\dxt \right)^{\nicefrac{1}{p}},
\]
which is warranted by conditions~\eqref{e.k-condition-NT} and~\eqref{e.k-condition3-NT}. Having at hand this choice of $k$, applying Lemma~\ref{t.alg-est-2} with $\alpha=1/(p-1)$, ~\eqref{e.bnd-est-11-NT} implies that 
\begin{align*}
\sup_{Q_{\sigma \rho, \sigma \theta}}u_+ &\leq 2^{\frac{(2-p)_+}{p-1}}\left(k+\frac{1}{2^{\frac{(2-p)_+}{p-1}}4}\mathsf{M}_0+\ell(0)\right)\\
&\leq 2^{\frac{(2-p)_+}{p-1}}\left[2\ell(0)+\frac{2}{2^{\frac{(2-p)_+}{p-1}}4}\mathsf{M}_0 + \delta \boldsymbol{\mathfrak{S}}+\frac{\overline{C}}{(1-\sigma)^{(d+p)q_\ast}}\overline{\boldsymbol{\cA}}_\delta\left(\biint_{Q_{\rho,\theta}}u_+^p\dxt \right)^{\nicefrac{1}{p}}\right].
\end{align*}
Here we crudely estimate, using Lemma~\ref{t.alg-est-2} with $\alpha=1/(p-1)$ again, that
\begin{align*}
\ell(0) &\leq 2^{\frac{(2-p)_+}{p-1}}\left(\boldsymbol{\Gamma}\int_{-\theta}^0 \int_{\R^d \setminus B_R} \frac{u_+^{p-1}(y,\tau)}{|y|^{d+sp}}\dytau\right)^{\frac{1}{p-1}}+2^{\frac{(2-p)_+}{p-1}}\underbrace{(2-p)_+^{\frac{1}{p-1}}}_{\leq 1} \frac{1}{4^{\frac{(2-p)_+}{p-1}}8}\mathsf{M}_0 \\
&\leq 2^{\frac{(2-p)_+}{p-1}}\left(\boldsymbol{\Gamma}\int_{-\theta}^0 \int_{\R^d \setminus B_R} \frac{u_+^{p-1}(y,\tau)}{|y|^{d+sp}}\dytau\right)^{\frac{1}{p-1}}+\frac{1}{2^{\frac{(2-p)_+}{p-1}}8}\mathsf{M}_0.
\end{align*}
Using this, the last estimate above becomes
\begin{align*}
\sup_{Q_{\sigma \rho, \sigma \theta}}u_+ &\leq 2^{\frac{(2-p)_+}{p-1}}\left[2^{\frac{(2-p)_+}{p-1}+1}\left(\boldsymbol{\Gamma}\int_{-\theta}^0 \int_{\R^d \setminus B_R} \frac{u_+^{p-1}(y,\tau)}{|y|^{d+sp}}\dytau\right)^{\frac{1}{p-1}}+\frac{3}{2^{\frac{(2-p)_+}{p-1}}4}\mathsf{M}_0  \right.\\
&\left. \quad \quad \quad \quad \quad \quad \quad+ \delta \boldsymbol{\mathfrak{S}}+ \frac{\overline{C}}{(1-\sigma)^{(d+p)q_\ast}}\overline{\boldsymbol{\cA}}_\delta\left(\biint_{Q_{\rho,\theta}}u_+^p\dxt \right)^{\nicefrac{1}{p}}\right]\\
&\leq \frac{3}{4}\mathsf{M}_0+\frac{C(p)}{(1-\sigma)^{(d+p)q_\ast}}\sfQ,
\end{align*}
where we used shorthand notation
\[
\sfQ:=\left(\boldsymbol{\Gamma}\int_{-\theta}^0 \int_{\R^d \setminus B_R} \frac{u_+^{p-1}(y,\tau)}{|y|^{d+sp}}\dytau\right)^{\frac{1}{p-1}}+\delta \boldsymbol{\mathfrak{S}}+\overline{C}\,\overline{\boldsymbol{\cA}}_\delta\left(\biint_{Q_{\rho,\theta}}u_+^p\dxt \right)^{\nicefrac{1}{p}}.
\]
Appealing to this over concentric cylinders $\sigma_1Q_{ \rho, \theta} \subset \sigma_2Q_{\rho, \theta}$ with $\sigma \leq \sigma_1<\sigma_2 \leq  1$, we deduce that
\[
\sup_{\sigma_1Q_{\rho, \theta}} u_+ \leq \frac{3}{4} \sup_{\sigma_2Q_{\rho, \theta}} u_+ +\frac{C(p)}{(\sigma_2-\sigma_1)^{(d+p)q_\ast}} \sfQ.
\]
A standard iteration argument (see~\cite[Lemma 6.1, Page 191]{Giusti}) then yields that
\[
\sup_{\sigma Q_{\rho, \theta}}u_+ \leq \frac{C}{(1-\sigma)^{(d+p)q_\ast}}\sfQ
\]
for a constant $C(p,q_\ast)<\infty$. Note that the quantity $q_\ast$ depends only on $(d,s,p)$.

Now, fix an arbitrary exponent $\nu \in (0,p)$. After replacing $\sigma$ by $\nicefrac{(1+\sigma)}{2}$ in the above estimate, appealing to Young's inequality and rearranging yield that, for every $Q_{R,\theta} \subset \Omega_T$, $\sigma \in (0,1)$ and $\rho \in (0,R)$, 
\begin{align*}
\sup_{\frac{1+\sigma}{2}Q_{ \rho, \theta}}u_+
&\leq \frac{c}{(1-\sigma)^{(d+p)q_\ast}}\delta\left[1+\left(\frac{R}{\rho}\right)^{\frac{d}{p-1}}\right]\sup_{Q_{R,\theta}} u_+ \notag\\
&\quad \quad  +\frac{c}{(1-\sigma)^{(d+p)q_\ast}}\left[\left(\frac{R}{R-\rho}\right)^{d+sp} \int_{-\theta}^0 \int_{\R^d \setminus B_R}\frac{u_+^{p-1}}{|x|^{d+sp}}\dxt \right]^{\frac{1}{p-1}} \notag\\
&\quad \quad  +C^{\frac{p}{\nu}}\delta^{-\frac{p-\nu}{\nu}}\overline{\boldsymbol{\cA}}^{\frac{p}{\nu}}_\delta\left[\frac{1}{(1-\sigma)^{(d+p)q_\ast p}}\biint_{Q_{\rho,\theta}}u_+^\nu\dxt \right]^{\nicefrac{1}{\nu}}, 
\end{align*}
where $c(d,s,p)<\infty$ and $C(\data)<\infty$. With this estimate at hand, for the remainder part, the argument developed in~\cite[Section 3.2, Step 4]{CN26} applies verbatim by choosing small enough $\delta \in (0,1)$ to depend on $\sigma^{d/p}$ and $\data$ without any extra adjustment, and therefore completes the proof.
\end{proof}

As a corollary of Theorem~\ref{t.boundedness-NT}, we also have the following local boundedness. 

\begin{corollary}\label{t.boundedness-NT-cor}
Assume the same assumption as in Theorem~\ref{t.boundedness-NT}. Let $\tau \in (0,T)$ be and arbitrary time. There exists a Lebesgue instant $\theta \in (0,\tau)$ such that for $B_\rho(x_0) \times [\theta, \tau] \Subset \Omega_T$, the following estimate is valid:
\begin{align*}
\sup_{B_{\sigma \rho}(x_0) \times [\theta, \tau]}u_\pm &\leq \frac{C_\ast}{(1-\sigma)^{(d+p)q_\ast+\frac{d+sp}{p-1}}}\left[ \frac{\tau-\theta}{(\sigma \rho)^{sp}}\dashint_{\theta}^{\tau} \tail \big(u_\pm(t)\,; B_{\sigma \rho}(x_0)\big)^{p-1}\d{t}\right]^{\nicefrac{1}{(p-1)}}\\
&\quad \quad \quad \quad \quad \quad \quad +C_\ast \left[\frac{\sfA}{(1-\sigma)^{(d+p)q_\ast p}}\biint_{B_\rho(x_0) \times [\theta, \tau]}u_\pm^{p-1}\dxt \right]^{\nicefrac{1}{(p-1)}}
\end{align*}
with a constant $C_\ast(\data)<\infty$ and an exponent $q_\ast(d,s,p)>0$,  where we denoted
\[
\sfA:=\left(\frac{\rho^{sp}}{\tau-\theta}+\dfrac{1}{\sigma^{d+sp}}\right)^{pq_\ast}\left(\dfrac{\tau-\theta}{\rho^{sp}}\right) \quad \mbox{and} \quad
q_\ast:=\begin{cases}
\frac{1}{p}\left(\frac{2}{p \wedge 2}\frac{d}{sp}+1\right), \quad &\mbox{if}\quad sp<d,\\[2mm]
\frac{1}{p}\left(\frac{4}{p \wedge 2}+1\right), \quad &\textrm{if}\quad sp\geq d.
\end{cases}
\]

\end{corollary}
\begin{proof}
The proof is done by essentially the same argument as in Theorem~\ref{t.boundedness-NT}. Indeed, we need to do is slightly changing the set up: Fix $\sigma \in (0,1)$, $\rho \in (0,\infty)$ and $\tau \in (0,T)$. Let $\theta \in (0,\tau)$ be an admissible initial time level to be specified later, and define $I:=[\theta, \tau]$. 

For $i \in \N_0$, define $\rho_i:=\sigma \rho+2^{-i}(1-\sigma) \rho$. Consider families of cylinders $\{Q_i\}_{i \in \N_0}$, $\{\widetilde{Q}_i\}_{i \in \N_0}$, $\{\widehat{Q}_i\}_{i \in \N_0}$ and $\{\overline{Q}_i\}_{i \in \N_0}$ given by
\[
Q_i:=B_i \times I, \quad \widetilde{Q}_i:=\widetilde{B}_i \times I, \quad \widehat{Q}_i:=\widehat{B}_i \times I \quad \mbox{and} \quad \overline{Q}_i:=\overline{B}_i \times I.
\]
We then select a smooth cutoff function, independent of time, $\xi=\xi(x):\R^d\to [0,1]$ satisfying
\[
\1_{\widetilde{B}_i} \leq \xi \leq \1_{\widehat{B}_i}, \quad |\nabla \xi| \leq \frac{2^{i+4}}{(1-\sigma)\rho}.
\]
Note that $\partial_t\xi^p =0$. Basically, replacing the cutoff function $\zeta$ as presented in the proof of Theorem~\ref{t.boundedness-NT} by $\xi$ we will demonstrate the same argument as Steps 1--4 with some modifications. To see it, define
\[
\sfG(t) :=\int_{B_i \times \{t\}}\mathfrak{g}_+(u, k(t))\d{x} \quad \mbox{for} \quad t \in I.
\]
Since $\sfG (\cdot) \in L^1(I)$, the set
\[
E:=\left\{t \in I : \sfG (t) \leq  2 \dashint_I \sfG (\tau) \d{\tau}\right\}
\]
has positive measure. Indeed, otherwise
\[
\sfG(t) > 2 \dashint_I \sfG(\tau) \d{\tau}
\]
for almost every $t \in I$, and taking the integral average over $I$ gives
\[
\dashint_I\sfG(t) \d{t}> 2 \dashint_I \sfG(t) \d{t};
\]
a contradiction. By the $L^p$-continuity of weak solutions as in Definition~\ref{def-of-NT}, every time level is a Lebesgue instant for $\sfG(\cdot)$; more precisely, due to the $L^p$-continuity of weak solutions $u$, and the continuity of $\mathfrak{g}_+$, one can easily check that $\sfG$ is continuous in time, namely, if $t_0$ is a Lebesgue instant for $u$, then a direct manipulation verifies that
\[
\lim\limits_{h \to 0}\dashint_{t_0}^{t_0+h}|\sfG(t)-\sfG(t_0)|\d{t}=0. 
\]
Hence there exists $\theta \in E$ such that
\begin{equation}\label{e.boundedness-NT-cor-1}
\int_{B_i \times \{\theta\}}\mathfrak{g}_+(u, k(t))\d{x} \leq  \frac{2}{|I|} \iint_{B_i \times I}\mathfrak{g}_+(u, k(t))\dxt.
\end{equation}
Fix such a choice of $\theta$ once and for all. Notice that the boundary term 
\[
\int_{B_i \times \{\theta\}}\mathfrak{g}_+(u, k(t))\d{x}
\]
appears on the right side of~\eqref{e.bnd-est-1-NT} in place of $\sfI$. Thus, by~\eqref{e.boundedness-NT-cor-1} the rest of the proof is almost a verbatim repetition of the proof of Theorem~\ref{t.boundedness-NT}, while keeping the time interval $I=[\theta,\tau]$. Note also that, in the procedure of the proof, we require the corresponding qualitative local boundedness in place of Proposition~\ref{t.qualitative-bnd-NT}, which is proved similarly using~\eqref{e.boundedness-NT-cor-1} again.
This proves the statement.
\end{proof}

\section{Measure theoretical properties}\label{Sect.4}
In this section, we provide measure theoretical lemmas for~\eqref{e.NT}. Many of these estimates can be straightforwardly obtained by suitably adapting~\cite[Section 4]{CN26}.

\begin{lemma}[De Giorgi type lemma]\label{t.DGtype-NT}
Let $p>1,s \in (0,1)$. Given $z_0=(x_0,t_0) \in \Omega_T$, assume that $Q_{2\rho,  2\tau }(z_0) \subset \mathcal{Q}_{R}\Subset \Omega_T$, where $\tau:=\delta \rho^{sp}$ with an arbitrary parameter $\delta \in (0,1]$. Let $u$ be a weak super-solution to~\eqref{e.NT}-\eqref{e.kernel}, in the sense of Definition~\ref{def-of-NT}, satisfying $u \geq 0$ in $\mathcal{Q}_R$ and let $\eps \in (0,\infty)$. There exists $\nu \in (0,1)$, depending only on $\data$ and $\eps$ such that, if
\[
\left(\frac{\rho}{R}\right)^{\frac{sp\eps}{(p-1)(p-1+\eps)}} \left(\dashint_{I_R(t_0)}\tail\big(u_-(t)\,; B_R(x_0)\big)^{p-1+\eps}\d{t}\right)^{\nicefrac{1}{(p-1+\eps)}} \leq \frac{k}{2}
\]
and
\[
\left|\{u<k\} \cap Q_{2\rho,  2\tau }(z_0)\right| \leq \nu \left|Q_{2\rho,  2\tau }(z_0)\right|
\]
then 
\[
u \geq \frac{k}{2} \quad \textrm{a.e.\,\,in}\,\, \,Q_{\rho, \tau }(z_0).
\]
Moreover,  we can track  $\nu=\nu_0 \delta^q$ for some constants $\nu_0 \in (0,1)$ and $q>1$, both depending only on $\data$ and $\eps$.
\end{lemma}

\begin{proof}
The proof is a modification of~\cite[Lemma 4.1]{CN26}. While our approach may appear similar to~\cite{SZ25} at first glance, it differs in several subtle but crucial aspects. 

\emph{Step 1: Set up for the iteration.} We begin with the setup for De Giorgi's iteration. Define a shrinking family of cylinders $Q_i:=Q_{\rho_i,\tau_i}:=B_i \times T_i$ for $i \in \N_0$, where $\rho_i$ and $\tau_i$ are given by
\[
\rho_i:=(1+2^{-i})\rho, \quad \tau_i:=\delta (1+2^{-i})\rho^{sp}
\]
with an arbitrary parameter $\delta \in (0,1]$ and we henceforth  write $B_i:= B_{\rho_i}$ and $T_i := (-\tau_i,0]$ for short. Moreover, define
\[
\begin{cases}
\widetilde{\rho}_i:=\dfrac{1}{2}(\rho_i+\rho_{i+1}), \quad \widetilde{\tau}_i:=\dfrac{1}{2}(\tau_i+\tau_{i+1}), \\[4pt]
\widetilde{B}_i:=B_{\widetilde{\rho}_i},  \quad \widetilde{T}_i:=(-\widetilde{\tau}_i,0]\quad \mbox{and}\quad \widetilde{Q}_i:=\widetilde{B}_i\times \widetilde{T}_i.
\end{cases}
\]
Given $k \in [0,\infty)$, we define decreasing sequences of levels
\[
k_i:=\frac{k}{2}+\frac{k}{2^{i+1}}, \quad \widetilde{k}_i:=\frac{1}{2}(k_i+k_{i+1}).
\]
In addition, it is convenient to introduce the notation
\[
A_i:=\{u <k_i\} \cap Q_i \quad \mbox{and}\quad \widetilde{A}_i:=\{u <\widetilde{k}_i\} \cap Q_i
\]
and therefore, clearly, $\widetilde{A}_i \subset A_i$. Moreover, define the time-slice of $A_i$ as
\[
A_i(t):=\left\{u(\cdot, t) <k_i \right\} \cap B_i \quad \mbox{for} \quad t \in T_i.
\]
\emph{Step 2: Carry out the iterative procedure.} Let $\varphi \in C^\infty(\R^{d+1} ; [0,1])$ be a cutoff function to satisfy
\[
\1_{Q_{i+1}} \leq \varphi \leq \1_{Q_i}, \quad |\nabla \varphi| \leq 2^{i+4}/\rho, \quad \supp(\varphi) \subset \widetilde{Q}_i. 
\]
An application of Proposition~\ref{FS} with $\mathsf{d}=2^{-i-4}$ and H\"{o}lder's inequality imply that
\begin{align}\label{e.dg-est-NT-a}
\frac{k}{2^{i+3}}\left|A_{i+1}\right|
&\leq \iint_{\widetilde{Q}_i} \widetilde{w}_i\varphi\dxt \notag\\
&\leq \left[\iint_{\widetilde{Q}_i} (\widetilde{w}_i\varphi)^{p \kappa}\dxt\right]^{\nicefrac{1}{p\kappa}}\left|A_i\right|^{1-\nicefrac{1}{p\kappa}} \notag\\
&\leq C \Bigg[\rho^{sp} \int_{\widetilde{T}_i}\int_{\widetilde{B}_i}\int_{\widetilde{B}_i} \frac{\left|\varphi \widetilde{w}_i(x,t)-\varphi \widetilde{w}_i(y,t)\right|^p}{|x-y|^{d+sp}}\dxyt \notag\\
&\quad \quad \quad +2^{i(d+sp)}\iint_{\widetilde{Q}_i}(\varphi \widetilde{w}_i)^p\dxt\Bigg]^{\nicefrac{1}{p\kappa}}\left[\sup_{t _\in \widetilde{T}_i}\dashint_{\widetilde{B}_i}\widetilde{w}_i^p(t)\d{x} \right]^{\frac{\kappa_\ast-1}{p\kappa \kappa_\ast}}\left|A_i\right|^{1-\nicefrac{1}{p\kappa}} 
\end{align}
with $\kappa=1+\frac{\kappa_\ast-1}{\kappa_\ast}$. Here, to lighten the notation, we denoted $\widetilde{w}_i:=(u-\widetilde{k}_i)_-$. 

We split the estimate of the integral average on the right side into $p \geq 2$ and $p \in (1,2)$. 

In the case $p \geq 2$, using Lemma~\ref{t.g}, we have
\begin{align*}
\sup_{t _\in \widetilde{T}_i}\dashint_{\widetilde{B}_i}\widetilde{w}_i^p(t)\d{x}  &\leq \sup_{t _\in \widetilde{T}_i}\dashint_{\widetilde{B}_i}\Big(|u|+|\widetilde{k}_i|\Big)^{p-2}(u-\widetilde{k}_i)_-^2\d{x}\\
&\leq \frac{1}{c_1(p)|\widetilde{B}_i|}\int_{\widetilde{B}_i \times \{t\}}\mathfrak{g}_-(u,\widetilde{k}_i)\d{x}.
\end{align*}
We next consider the case $p \in (1,2)$. Using the triangle inequality and H\"{o}lder's inequality with $\left(\nicefrac{2}{p}, \nicefrac{2}{(2-p)}\right)$ and appealing to Lemma~\ref{t.g} and~\eqref{e.bnd-est-4-NT} yields, we have
\begin{align*}
\sup_{t \in (-\widetilde{\theta}_i,0]}\,&\dashint_{\widetilde{B}_i \times \{t\}}\tilde{w}_{i}^p\d{x} \\[4pt]
&=\sup_{t \in (-\widetilde{\theta}_i,0]}\dashint_{\widetilde{B}_i \times \{t\}} \Big(|u|+|\widetilde{k}_i|\Big)^{\nicefrac{p(p-2)}{2}}(u-\widetilde{k}_i)_-^p \cdot \Big(|u|+|\widetilde{k}_i|\Big)^{\nicefrac{p(2-p)}{2}}\d{x} \\[4pt]
&\leq \left[\sup_{t \in (-\widetilde{\theta}_i,0] }\dashint_{\widetilde{B}_i \times \{t\}}\Big(|u|+|\widetilde{k}_i|\Big)^{p-2}(u-\widetilde{k}_i)_-^2\d{x} \right]^{\nicefrac{p}{2}} \\
&\quad \quad \quad \quad \times \left[\sup_{t \in (-\widetilde{\theta}_i,0] } \dashint_{\widetilde{B}_i \times \{t\}}\Big(|u|+|\widetilde{k}_i|\Big)^p \1_{\{0 \leq u(\cdot,t) \leq \widetilde{k}_i\}}\d{x} \right]^{\nicefrac{(2-p)}{2}}\\[4pt]
&\leq \frac{2^{\nicefrac{p(2-p)}{2}}}{c_1^{\nicefrac{p}{2}}}\left[\frac{1}{|\widetilde{B}_i|}\sup_{t \in (-\widetilde{\theta}_i,0] }\int_{\widetilde{B}_i \times \{t\}}\mathfrak{g}_-(u,\widetilde{k}_i)\d{x} \right]^{\nicefrac{p}{2}}k^{\nicefrac{p(2-p)}{2}}
\end{align*}
with $C(\data)<\infty$ in the case $1<p<2$.
As a result, we have obtained, using~\eqref{e.caccioppoli3}$_-$, that
\begin{align}\label{e.dg-est-NT-b}
\sup_{t \in (-\widetilde{\theta}_i,0]}\,&\dashint_{\widetilde{B}_i \times \{t\}}\tilde{w}_{i}^p\d{x} \notag\\
&\leq C\frac{k^{\frac{p(2-p)_+}{2}}}{|\widetilde{B}_i|^{\frac{p \wedge 2}{2}}}\left[\sup_{t \in (-\widetilde{\theta}_i,0] }\int_{\widetilde{B}_i \times \{t\}}\mathfrak{g}_-(u,\widetilde{k}_i)\d{x} \right]^{\frac{p \wedge 2}{2}} \notag\\
&\!\!\!\!\stackrel{\eqref{e.caccioppoli3}_-}{\leq} C\frac{k^{\frac{p(2-p)_+}{2}}}{|\widetilde{B}_i|^{\frac{p \wedge 2}{2}}}\left[\frac{1}{\tau_i-\widetilde{\tau}_i}\iint_{Q_i}\mathfrak{g}_-(u,\widetilde{k}_i)\dxt+\frac{\rho_i^{(1-s)p}}{(\rho_i-\widetilde{\rho}_i)^p}\iint_{Q_i}\widetilde{w}_i^p\dxt \right. \notag\\
&\left.  \quad \quad \quad \quad \quad  \quad \quad +\frac{\rho_i^{d}}{(\rho_i-\widetilde{\rho}_i)^{d+sp}}\iint_{Q_i}\widetilde{w}_i\d{x}\cdot \tail\big(\widetilde{w}_i(t) ; B_i \big)^{p-1}\d{t}\right]^{\frac{p \wedge 2}{2}} \notag \\
&\leq C\frac{k^{\frac{p(2-p)_+}{2}}}{|\widetilde{B}_i|^{\frac{p \wedge 2}{2}}}2^{i{\frac{p(p \wedge 2)}{2}}}\rho^{-{\frac{sp(p \wedge 2)}{2}}}\big[ \sfG + \sfH + \sfI \big]^{\frac{p \wedge 2}{2}},
\end{align}
where
\begin{align*}
\sfG&:=\frac{1}{\delta}\iint_{Q_i}\mathfrak{g}_-(u,\widetilde{k}_i)\dxt, \quad \sfH:=\iint_{Q_i}\widetilde{w}_i^p\dxt \quad \mbox{and}\\
\sfI&:=\iint_{Q_i}\widetilde{w}_i\d{x}\cdot \tail\big(\widetilde{w}_i(t) ; B_i \big)^{p-1}\d{t}.
\end{align*}
To obtain the last line from the penultimate one, we used 
\begin{equation}\label{e.dg-est-NT-computation}
\frac{\rho_i^{(1-s)p}}{(\rho_i-\widetilde{\rho}_i)^p}, \,\,\frac{\rho_i^{d}}{(\rho_i-\widetilde{\rho}_i)^{d+sp}} \leq \frac{c(d,s,p)2^{ip}}{\rho^{sp}}, \quad \frac{1}{\tau_i-\widetilde{\tau}_i} \leq \frac{c(d,s,p)2^{ip}}{\delta \rho^{sp}}.
\end{equation}
On the other hand, following the argument leading to~\cite[Lemma 4.1, (4.4)]{CN26} we are able to estimate
\begin{equation}\label{e.dg-est-NT-c}
\rho^{sp} \int_{\widetilde{T}_i}\int_{\widetilde{B}_i}\int_{\widetilde{B}_i} \frac{\left|\varphi \widetilde{w}_i(x,t)-\varphi \widetilde{w}_i(y,t)\right|^p}{|x-y|^{d+sp}}\dxt \leq C2^{ip}\big[ \sfG + \sfH + \sfI \big].
\end{equation}
for a constant $C(\data)<\infty$. Inserting these bounds~\eqref{e.dg-est-NT-b} and~\eqref{e.dg-est-NT-c} into the display~\eqref{e.dg-est-NT-a} yields
\begin{align}\label{e.dg-est-NT-d}
\frac{k}{2^{i+3}}\left|A_{i+1}\right| &\leq C\Big[2^{i(d+p)}\big[ \sfG + \sfH + \sfI \big] \Big]^{\frac{1}{p\kappa}}\left[\frac{k^{\frac{p(2-p)_+}{2}}}{|\widetilde{B}_i|^{\frac{p \wedge 2}{2}}}2^{i{\frac{p(p \wedge 2)}{2}}}\rho^{-{\frac{sp(p \wedge 2)}{2}}}\big[ \sfG + \sfH + \sfI \big]^{\frac{p \wedge 2}{2}}\right]^{\frac{\kappa_\ast-1}{p\kappa \kappa_\ast}}\left|A_i\right|^{1-\frac{1}{p\kappa}} \notag\\
&\leq  C2^{i\left(\frac{d+p}{p\kappa}+\frac{p \wedge 2}{2}\cdot \frac{\kappa_\ast-1}{\kappa \kappa_\ast}\right)}\left[\frac{k^{\frac{p(2-p)_+}{2}}}{|B_i|^{\frac{p \wedge 2}{2}}}\right]^{\frac{\kappa_\ast-1}{p\kappa \kappa_\ast}}\frac{\big[ \sfG + \sfH + \sfI \big]^{\frac{1}{p\kappa}+\frac{p \wedge 2}{2}\cdot \frac{\kappa_\ast-1}{p\kappa \kappa_\ast}}}{\rho^{\frac{s(p \wedge 2)}{2}\cdot \frac{\kappa_\ast-1}{\kappa \kappa_\ast}}}\left|A_i\right|^{1-\frac{1}{p\kappa}},
\end{align}
where in the last line we used $|B_i|/|\widetilde{B}_i| \leq c(d)$.
%
%
%

%
%
In order to estimate $\sfG$ on the right side of~\eqref{e.dg-est-NT-d}, we need to consider two cases: $p \geq 2$ and $p \in (1,2)$. When considering the case $p \geq 2$, by Lemma~\ref{t.g} and the fact that $0\leq u \leq \widetilde{k}_i \leq k$ on $\widetilde{A}_i$, we get
\begin{align*}
\sfG &\leq \frac{C}{\delta} \iint_{Q_i} \left(|u|+|\widetilde{k}_i|\right)^{p-2}(u-\widetilde{k}_i)_-^2 \dxt\\
&\leq \frac{C}{\delta} k^{p-2}  \iint_{Q_i} (u-\widetilde{k}_i)_-^2 \dxt\leq \frac{C}{\delta} k^{p}|A_i|. 
\end{align*}
with $C(p)<\infty$. On the other hand, in the case $1<p<2$ using $|u|+|\widetilde{k}_i| \geq (u-\widetilde{k}_i)_-$ implies that
\[
\sfG \leq \frac{C}{\delta} \iint_{Q_i}(u-\widetilde{k}_i)_-^p \dxt \leq \frac{C}{\delta} k^{p}|A_i|.
\]
In all cases, we have shown that
\begin{equation}\label{e.dg-est-NT-e} 
\sfG \leq \frac{C}{\delta} k^p|A_i|.
\end{equation}
We straightforwardly compute that
\begin{equation}\label{e.dg-est-NT-f}
\sfH \leq Ck^p|A_i|,
\end{equation}
while, a routine decomposing of the tail term into integrals on annuli yields
\[
\tail\big(\widetilde{w}(t)\,;B_i\big)^{p-1} \leq ck^{p-1}+c\left(\frac{\rho}{R}\right)^{sp}\tail\big((u(t)-k)_-\,;B_R\big)^{p-1}
\]
with $c(d,s,p)<\infty$, and therefore the H\"{o}lder inequality with powers $\left(\frac{p-1+\eps}{p-1}, \frac{p-1+\eps}{\eps}\right)$ implies
\begin{align*}
\sfI &\leq ck^{p-1}\iint_{Q_i}\widetilde{w}_i\dxt +c\left(\frac{\rho}{R}\right)^{sp} \iint_{Q_i}\widetilde{w}_i\cdot \tail\big((u(t)-k)_-\,;B_R\big)^{p-1}\d{t} \notag\\
&\leq ck^p\left|A_i\right|+c\rho^{\frac{(p-1)sp}{p-1+\eps}}k \left(\frac{\rho}{R}\right)^{\frac{\eps sp}{p-1+\eps}}\left(\dashint_{I_R}\tail\big(u_-(t)\,;B_R\big)^{p-1+\eps}\d{t} \right)^{\frac{p-1}{p-1+\eps}} \notag\\
&\quad \quad \quad \quad \quad \quad \quad \quad \quad \quad \times \left(\int_{T_i}|A_i(t)|^{\frac{p-1+\eps}{\eps}}\d{t}\right)^{\frac{\eps}{p-1+\eps}},
\end{align*}
where to obtain the last line we used
\[
\left(\frac{\rho}{R}\right)^{sp} R^{\frac{(p-1)sp}{p-1+\eps}}=\rho^{\frac{(p-1)sp}{p-1+\eps}}\left(\frac{\rho}{R}\right)^{\frac{\eps sp}{p-1+\eps}}.
\]
Now, if we enforce \[
\left(\frac{\rho}{R}\right)^{\frac{sp\eps}{(p-1)(p-1+\eps)}} \left(\dashint_{I_R}\tail \big(u_-(t)\,; B_R\big)^{p-1+\eps}\d{t}\right)^{\nicefrac{1}{(p-1+\eps)}} \leq \frac{k}{2}
\]
then the above last estimate turns out that
\begin{equation}\label{e.dg-est-NT-g}
\sfI \leq ck^p\left|A_i\right|+c\rho^{\frac{(p-1)sp}{p-1+\eps}}k^p\left(\int_{T_i}|A_i(t)|^{\frac{p-1+\eps}{\eps}}\d{t}\right)^{\frac{\eps}{p-1+\eps}}
\end{equation}
for a constant $c(d,s,p)<\infty$. Combining the preceding estimates~\eqref{e.dg-est-NT-e}--\eqref{e.dg-est-NT-g} yields
\begin{align*}
k\left|A_{i+1}\right| &\leq C \sfb^i \left[\frac{k^{\frac{p(2-p)_+}{2}}}{|B_i|^{\frac{p \wedge 2}{2}}}\right]^{\frac{\kappa_\ast-1}{p\kappa \kappa_\ast}}k^{\frac{1}{\kappa}+\frac{p \wedge 2}{2}\cdot \frac{\kappa_\ast-1}{\kappa \kappa_\ast}}\frac{\left|A_i\right|^{1-\frac{1}{p\kappa}}}{\delta^{\frac{p \wedge 2}{2}\cdot \frac{\kappa_\ast-1}{\kappa \kappa_\ast}}\rho^{\frac{s(p \wedge 2)}{2}\cdot \frac{\kappa_\ast-1}{\kappa \kappa_\ast}}} \\
&\quad \quad \times \left[\left|A_i\right|+\rho^{\frac{(p-1)sp}{p-1+\eps}}\left(\int_{T_i}|A_i(t)|^{\frac{p-1+\eps}{\eps}}\d{t}\right)^{\frac{\eps}{p-1+\eps}}\right]^{\frac{1}{p\kappa}+\frac{p \wedge 2}{2}\cdot \frac{\kappa_\ast-1}{p\kappa \kappa_\ast}}\\
&\leq C \sfb^i \frac{1}{|B_i|^{\frac{p \wedge 2}{2}\cdot \frac{\kappa_\ast-1}{p\kappa \kappa_\ast}}}\frac{k^{\frac{(2-p)_+}{2}\cdot \frac{\kappa_\ast-1}{\kappa \kappa_\ast}+\frac{1}{\kappa}+\frac{p \wedge 2}{2}\cdot \frac{\kappa_\ast-1}{\kappa \kappa_\ast}}}{\delta^{\frac{p \wedge 2}{2}\cdot \frac{\kappa_\ast-1}{\kappa \kappa_\ast}}\rho^{\frac{s(p \wedge 2)}{2}\cdot \frac{\kappa_\ast-1}{\kappa \kappa_\ast}}} \\
&\quad \quad \times \left[\left|A_i\right|^{1+\frac{p \wedge 2}{2}\cdot \frac{\kappa_\ast-1}{p\kappa \kappa_\ast}}+\left|A_i\right|^{1-\frac{1}{p\kappa}} \left[\rho^{\frac{(p-1)sp}{p-1+\eps}}\left(\int_{T_i}|A_i(t)|^{\frac{p-1+\eps}{\eps}}\d{t}\right)^{\frac{\eps}{p-1+\eps}}\right]^{\frac{1}{p\kappa}+\frac{p \wedge 2}{2}\cdot \frac{\kappa_\ast-1}{p\kappa \kappa_\ast}}\right]\\
&=C \sfb^i \frac{1}{|B_i|^{\frac{p \wedge 2}{2}\cdot \frac{\kappa_\ast-1}{p\kappa \kappa_\ast}}}\frac{k}{\delta^{\frac{p \wedge 2}{2}\cdot \frac{\kappa_\ast-1}{\kappa \kappa_\ast}}\rho^{\frac{s(p \wedge 2)}{2}\cdot \frac{\kappa_\ast-1}{\kappa \kappa_\ast}}} \\
&\quad \quad \times \left[\left|A_i\right|^{1+\frac{p \wedge 2}{2}\cdot \frac{\kappa_\ast-1}{p\kappa \kappa_\ast}}+\left|A_i\right|^{1-\frac{1}{p\kappa}} \left[\rho^{\frac{(p-1)sp}{p-1+\eps}}\left(\int_{T_i}|A_i(t)|^{\frac{p-1+\eps}{\eps}}\d{t}\right)^{\frac{\eps}{p-1+\eps}}\right]^{\frac{1}{p\kappa}+\frac{p \wedge 2}{2}\cdot \frac{\kappa_\ast-1}{p\kappa \kappa_\ast}}\right]
\end{align*}
where $\sfb:=2^{\frac{d+p}{p\kappa}+\frac{p \wedge 2}{2}\cdot \frac{\kappa_\ast-1}{\kappa \kappa_\ast}+1+(2-p)_++p}>1$. Here to obtain the last line from the penultimate one, using the definition of $\kappa$ given by Lemma~\ref{FS}, the power of $k$ was computed as
\[
\frac{(2-p)_+}{2}\cdot \frac{\kappa_\ast-1}{\kappa \kappa_\ast}+\frac{1}{\kappa}+\frac{p \wedge 2}{2}\cdot \frac{\kappa_\ast-1}{\kappa \kappa_\ast}=1.
\]
Dividing the last display by $|Q_{i+1}|$ yields that
\begin{align*}
\frac{\left|A_{i+1}\right|}{\left|Q_{i+1}\right|} &\leq C \sfb^i \frac{\left|Q_i\right|^{1+\frac{p \wedge 2}{2}\cdot \frac{\kappa_\ast-1}{p\kappa \kappa_\ast}}}{|B_i|^{\frac{p \wedge 2}{2}\cdot \frac{\kappa_\ast-1}{p\kappa \kappa_\ast}}\rho^{\frac{s(p \wedge 2)}{2}\cdot \frac{\kappa_\ast-1}{\kappa \kappa_\ast}}\left|Q_{i+1}\right|}\frac{1}{\delta^{\frac{p \wedge 2}{2}\cdot \frac{\kappa_\ast-1}{\kappa \kappa_\ast}}} \\
&\quad \times \left[\left(\frac{\left|A_i\right|}{|Q_i|}\right)^{1+\frac{p \wedge 2}{2}\cdot \frac{\kappa_\ast-1}{p\kappa \kappa_\ast}}+\left(\frac{\left|A_i\right|}{|Q_i|}\right)^{1-\frac{1}{p\kappa}}\left[\frac{\rho^{\frac{(p-1)sp}{p-1+\eps}}}{\left|Q_i\right|}\left(\int_{T_i}|A_i(t)|^{\frac{p-1+\eps}{\eps}}\d{t}\right)^{\frac{\eps}{p-1+\eps}}\right]^{\frac{1}{p\kappa}+\frac{p \wedge 2}{2}\cdot \frac{\kappa_\ast-1}{p\kappa \kappa_\ast}}\right].
\end{align*}
It is straightforward to check that
\[
\frac{\left|Q_i\right|^{1+\frac{p \wedge 2}{2}\cdot \frac{\kappa_\ast-1}{p\kappa \kappa_\ast}}}{|B_i|^{\frac{p \wedge 2}{2}\cdot \frac{\kappa_\ast-1}{p\kappa \kappa_\ast}}\rho^{\frac{s(p \wedge 2)}{2}\cdot \frac{\kappa_\ast-1}{\kappa \kappa_\ast}}\left|Q_{i+1}\right|} =\frac{\left|Q_i\right|}{\left|Q_{i+1}\right|}\frac{\left|T_i\right|^{\frac{p \wedge 2}{2}\cdot \frac{\kappa_\ast-1}{p\kappa \kappa_\ast}}}{\rho^{\frac{s(p \wedge 2)}{2}\cdot \frac{\kappa_\ast-1}{\kappa \kappa_\ast}}} \leq C(d)\delta^{\frac{p \wedge 2}{2}\cdot \frac{\kappa_\ast-1}{p\kappa \kappa_\ast}}
\]
and
\[
\frac{\rho^{\frac{(p-1)sp}{p-1+\eps}}}{\left|Q_i\right|}\left(\int_{T_i}|A_i(t)|^{\frac{p-1+\eps}{\eps}}\d{t}\right)^{\frac{\eps}{p-1+\eps}} \leq C(s,p,\eps) \left[\int_{T_i} \left(\frac{|A_i(t)|}{|B_i|}\right)^{\frac{p-1+\eps}{\eps}}\d{t}\right]^{\frac{\eps}{p-1+\eps}}.
\]
Moreover, we set
\begin{align*}
\sfY_i:=\frac{|A_i|}{|Q_i|} \leq 1, \quad \mbox{and} \quad \sfZ_i:=\left[\dashint_{T_i} \left(\frac{|A_i(t)|}{|B_i|}\right)^{\frac{p-1+\eps}{\eps}}\d{t}\right]^{\frac{\eps}{(1+\lambda)(p-1+\eps)}} \leq 1.
\end{align*}
Combining these, we obtain
\begin{equation}\label{e.dg-est-NT-h}
\sfY_{i+1} \leq C \sfb^i\left(\sfY_{i}^{1+\beta}+\sfY_{i}^{\beta}\sfZ_i^{1+\lambda}\right), \quad \forall i \in \N_0 
\end{equation}
where $\beta:=\frac{p \wedge 2}{2}\cdot \frac{\kappa_\ast-1}{p\kappa \kappa_\ast}$ for short, and the parameter $\lambda \in (0,\infty)$ is still to be specified. Here we used $\beta<1-\frac{1}{p\kappa}$ and thus, $\sfY_i^{1-\frac{1}{p\kappa}} \leq \sfY_i^{\beta}$.
\medskip

\emph{Step 3: Run the iteration scheme.}  Next, our task is to deduce the recursive inequality for $\{\sfZ_i\}_{i \in\N_0}$. It is straightforward to check  that
\begin{align}\label{e.dg-est-NT-i}
\left|A_{i+1}(t)\right|&\leq \int_{A_{i+1}(t)}\left(\frac{\widetilde{k}_i-u}{\widetilde{k}_i-k_{i+1}}\right)^{p(1+\lambda)}\d{x} \notag\\
&\leq \left(\frac{2^{i+3}}{k}\right)^{p(1+\lambda)}\int_{B_{i+1}}(u-\widetilde{k}_i)_-^{p(1+\lambda)}\d{x}.
\end{align}
We appeal to Lemma~\ref{t.GN} with
\[
\widetilde{q}=p(1+\lambda), \quad \widetilde{r}=\frac{sp^2(1+\lambda)}{d\lambda}, \quad \mbox{and} \quad m=p,
\]
while we select the free parameter $\lambda>0$ so that
\[
\widetilde{r}=\frac{p(1+\lambda)(p-1+\eps)}{\eps} \quad \iff \quad \lambda=\frac{\eps sp}{d(p-1+\eps)}.
\]
Note that these selections yield that
\[
\frac{d\widetilde{q}}{sp\widetilde{q}+d\widetilde{r}}=\frac{1}{\widetilde{r}}\cdot\frac{d\widetilde{q}\widetilde{r}}{sp\widetilde{q}+d\widetilde{r}}=\frac{\eps} {p(1+\lambda)(p-1+\eps)}\cdot p=\frac{\eps} {(1+\lambda)(p-1+\eps)}.
\]
This, together with~\eqref{e.dg-est-NT-i}, the H\"{o}lder inequality and~\eqref{e.dg-est-NT-b}, yields that
\begin{align}\label{e.dg-est-NT-j}
\sfZ_{i+1}&=\left[\frac{1}{|B_{i+1}|^{\frac{p-1+\eps}{\eps}}|T_{i+1}|}\int_{T_{i+1}}|A_{i+1}(t)|^{\frac{p-1+\eps}{\eps}}\d{t}\right]^{\frac{\eps}{(1+\lambda)(p-1+\eps)}} \notag\\[4pt]
&\!\!\!\!\stackrel{\eqref{e.dg-est-NT-i}}{\leq} \left[\frac{1}{|B_{i+1}|^{\frac{p-1+\eps}{\eps}}|T_{i+1}|}\left(\frac{2^{i+3}}{k}\right)^{\frac{p(1+\lambda)(p-1+\eps)}{\eps}}\int_{T_{i+1}}\left\|(u-\widetilde{k}_i)_-\right\|_{L^{p(1+\lambda)}(B_{i+1})}^{\frac{p(1+\lambda)(p-1+\eps)}{\eps}}\d{t}\right]^{\frac{\eps}{(1+\lambda)(p-1+\eps)}}\notag\\[4pt]
&\!\!\leq C \frac{2^{ip}}{k^p} \left(\frac{1}{|B_{i+1}|^{\frac{p-1+\eps}{\eps}}|T_{i+1}|}\right)^{\frac{\eps}{(1+\lambda)(p-1+\eps)}}\left[\int_{\widetilde{T}_i}[(u-\widetilde{k}_i)_-]_{W^{s,p}(\widetilde{B}_i)}^p\d{t}\right. \notag\\[4pt]
&\left. \quad \quad \quad  \quad \quad \quad  \quad \quad \quad \quad\quad  +(\widetilde{\rho}_i)^{-sp}\left\|(u-\widetilde{k}_i)_-\right\|_{L^p(\widetilde{Q}_i)}^{p} +\sup_{t_ \in \widetilde{T}_i}\left\|(u-\widetilde{k}_i)_-\right\|_{L^{p}(\widetilde{B}_i)}^{p}\right].
\end{align}
We now consider two cases. When $p \geq 2$, using Lemma~\ref{t.g} and $k_i>\widetilde{k}_i$, we have
\begin{align*}
\mathfrak{g}_-(u,k_i) &\geq c_1(p) \Big(|u|+|k_i|\Big)^{p-2}(u-k_i)_-^2\geq c_1(p)(u-\widetilde{k}_i)_-^p.
\end{align*}
When dealing with $1<p<2$, since $0 \leq u \leq k_i \leq k$ on $\widetilde{A}_i$ it holds that $|u|+|k_i| \leq 2k$, and therefore we carefully estimate, using Lemma~\ref{t.g} and $k_i>\widetilde{k}_i$ again, that
\begin{align*}
\mathfrak{g}_-(u,k_i) &\geq c_1(p) \Big(|u|+|k_i|\Big)^{p-2}(u-k_i)_-^2 \\
&\geq c_1(p) (2k)^{p-2}(u-k_i)_-^{2-p}\cdot (u-k_i)_-^{p}\\
&\geq c_1(p) (2k)^{p-2}(k_i-\widetilde{k}_i)^{2-p}\cdot (u-\widetilde{k}_i)_-^{p}\\
&= \frac{c_1(p)}{2^{(i+4)(2-p)}} (u-\widetilde{k}_i)_-^{p}.
\end{align*}
On the whole, 
\[
\sup_{t_ \in \widetilde{T}_i}\left\|(u-\widetilde{k}_i)_-\right\|_{L^{p}(\widetilde{B}_i)}^{p} \leq C(p)2^{i(2-p)_+}\sup_{t \in \widetilde{T}_i}\int_{\widetilde{B}_i \times \{t\}}\mathfrak{g}_-(u,k_i)\d{x}.
\]
After invoking Lemma~\ref{t.alg-est-3}, we will use this,~\eqref{e.dg-est-NT-b},~\eqref{e.caccioppoli3} and~\eqref{e.dg-est-NT-computation}. Thus, the bracket of~\eqref{e.dg-est-NT-j} is estimated as 
\begin{align*}
\Big[ \cdots \Big] & \quad \leq \int_{\widetilde{T}_i}[(u-k_i)_-]_{W^{s,p}(\widetilde{B}_i)}^p\d{t}+ (\widetilde{\rho}_i)^{-sp}\left\|(u-k_i)_-\right\|_{L^p(\widetilde{Q}_i)}^{p} \\
& \quad \quad \quad +C(p)2^{i(2-p)_+}\sup_{t \in \widetilde{T}_i}\int_{\widetilde{B}_i \times \{t\}}\mathfrak{g}_-(u,k_i)\d{x} \\
&\!\! \stackrel{\eqref{e.caccioppoli3}, \eqref{e.dg-est-NT-computation}}{\leq} C2^{i[(2-p)_++p]}\rho^{-sp}\big[\sfG^\prime + \sfH^\prime + \sfI^\prime \big]
\end{align*}
for a constant $C(\data)<\infty$, where we set
\begin{align*}
\sfG^\prime &:=\frac{1}{\delta}\iint_{Q_i}\mathfrak{g}_-(u,k_i)\dxt, \quad \sfH^\prime:=\iint_{Q_i}w_i^p\dxt \quad \mbox{and}\\
\sfI^\prime&:=\iint_{Q_i}w_i\d{x}\cdot \tail\big(w_i(t) ; B_i \big)^{p-1}\d{t}.
\end{align*}
Applying the same argument with $\sfG^\prime+\sfH^\prime+\sfI^\prime$ in place of $\sfG+\sfH+\sfI$, we conclude that
\begin{align*}
\sfG^\prime + \sfH^\prime + \sfI^\prime &\leq \frac{C}{\delta} k^p\left[\left|A_i\right|+\rho^{\frac{(p-1)sp}{p-1+\eps}}\left(\int_{T_i}|A_i(t)|^{\frac{p-1+\eps}{\eps}}\d{t}\right)^{\frac{\eps}{p-1+\eps}}\right]\\
& \leq \frac{C|Q_i|}{\delta}k^p\left[\frac{|A_i|}{|Q_i|} +\left[\dashint_{T_i}\left(\frac{|A_i(t)|}{|B_i|}\right)^{\frac{p-1+\eps}{\eps}}\d{t}\right]^{\nicefrac{\eps}{(p-1+\eps)}}\right]\\
&=\frac{C|Q_i|}{\delta}k^p \left(\sfY_i+\sfZ_i^{1+\lambda}\right).
\end{align*}
Substituting back to~\eqref{e.dg-est-NT-j} gives
\begin{align}\label{e.dg-est-NT-k}
\sfZ_{i+1} &\leq C 2^{i[(2-p)_++p]}\left(\frac{1}{|B_{i+1}|^{\frac{p-1+\eps}{\eps}}|T_{i+1}|}\right)^{\frac{\eps}{(1+\lambda)(p-1+\eps)}}\frac{|Q_i|}{\delta \rho^{sp}} \left(\sfY_i+\sfZ_i^{1+\lambda}\right) \notag\\
& \leq C \sfb^i \delta^{-\frac{\eps d}{d(p-1+\eps)+\eps sp}} \left(\sfY_i+\sfZ_i^{1+\lambda}\right).
\end{align}
To obtain the last line we used, by $\lambda=\frac{\eps sp}{d(p-1+\eps)}$, that
\[
\frac{d\lambda}{1+\lambda}-\frac{\eps sp}{(1+\lambda)(p-1+\eps)}=0, \quad \mbox{and} \quad \frac{\eps}{(1+\lambda)(p-1+\eps)}= \frac{\eps d}{d(p-1+\eps)+\eps sp}
\]
and hence,
\begin{align*}
\left(\frac{1}{|B_{i+1}|^{\frac{p-1+\eps}{\eps}}|T_{i+1}|}\right)^{\frac{\eps}{(1+\lambda)(p-1+\eps)}}\frac{|Q_i|}{\delta \rho^{sp}}
&\leq C\delta^{-\frac{\eps d}{d(p-1+\eps)+\eps sp}}
\end{align*}
with $C(d,p,\lambda, \eps)<\infty$. Let
\[
\sfK_i:=\sfY_i+\sfZ_i^{1+\lambda}, \quad \mbox{and} \quad \gamma:=\frac{\eps d}{d(p-1+\eps)+\eps sp}.
\]
Thus,~\eqref{e.dg-est-NT-h} and~\eqref{e.dg-est-NT-k} yield, following the argument of~\cite[Chapter I.4, Lemma 4.2]{DiB93}, that
\[
\sfK_{i+1} \leq 2C^{1+\lambda} \delta^{-\gamma(1+\lambda)}\sfb^{(1+\lambda)i}\sfK_i^{1+(\lambda \wedge \beta)}, \quad \forall i \in \N_0.
\]
By Lemma~\ref{t.FGC}, there exists a positive constant $\nu \in (0,1)$ depending only on $\data$ and $\eps>0$, such that $\sfK \to 0$ as $i \to \infty$, if we enforce that $\boldsymbol{K}_0 \leq \nu$, thereby implying
\[
|A_0|=\left|\left\{u<k\right\} \cap Q_{2\rho,2\tau}\right| \leq \nu\left|Q_{2\rho,2\tau}\right|.
\]
In particular,
\[
\sfK_\infty=0  \quad \Longrightarrow \quad u \geq \frac{k}{2} \quad \mbox{a.e. in $Q_{\rho,\tau}$},
\]
proving the statement.
\end{proof}

The following lemma states that point-wise information at later times under a critical measure density in a ball at some time.
\begin{lemma}\label{t.mt-NT}
Let $u$ be a weak super-solution to~\eqref{e.NT}-\eqref{e.kernel}, in the sense of  Definition~\ref{def-of-NT}, satisfying $u \geq 0$ in $\mathcal{Q}_R$ and let $\eps \in (0,\infty]$. Suppose, for some constants $k>0$ and $\alpha \in (0,1]$, that
\[
\Big|\left\{u(\cdot,t_0) \geq k \right\} \cap B_\rho(x_0)\Big| \geq \alpha |B_\rho(x_0)|.
\]
There exist constants $\eta$ and $\delta$ in $(0,1)$, both depending only on $\data$ and $\alpha$, such that either
\[
\left(\frac{\rho}{R}\right)^{\frac{sp\eps}{(p-1)(p-1+\eps)}} \left(\dashint_{I_R(t_0)} \tail \big(u_-(t)\,;B_R(x_0)\big)^{p-1+\eps}\d{t}\right)^{\nicefrac{1}{(p-1+\eps)}}>\eta k
\]
or
\[
\Big|\left\{u(\cdot,t) \geq \eta k \right\} \cap B_\rho(x_0)\Big| \geq \frac{\alpha}{2} |B_\rho(x_0)|
\]
for every $t \in (t_0,t_0+\delta \rho^{sp}]$, provided that  $B_\rho(x_0)\times (t_0,t_0+\delta \rho^{sp}] \subset \cQ_R \Subset \Omega_T$. In more generality, we fave
\[
\eta =\alpha/(8p) \quad \mbox{and} \quad \delta =\delta_0\alpha^{(d+p+1)\frac{p-1+\eps}{\eps}}
\]
for a constant $\delta_0=\delta_0(\data, \eps) \in (0,1)$.
\end{lemma}

\begin{proof}
The proof is crudely similar to the one of~\cite[Lemma 4.2]{CN26}. As usual, we may let $(x_0,t_0)=(0,0)$ and it suffices to check that
\[
\Big|\left\{u(\cdot,t) <\eta k \right\} \cap B_\rho\Big| \leq \left(1-\frac{\alpha}{2}\right) |B_\rho|, \quad \forall t \in (0, \delta \rho^{sp}],
\]
where $\eta \in (0,\nicefrac{1}{2})$ is to be selected below, provided that 
\begin{equation}\label{e.mt1-est-NT-1}
\left(\frac{\rho}{R}\right)^{\frac{sp\eps}{(p-1)(p-1+\eps)}} \left(\dashint_{I_R} \tail \big(u_-(t)\,;B_R\big)^{p-1+\eps}\d{t}\right)^{\nicefrac{1}{(p-1+\eps)}} \leq \eta k.
\end{equation}
We remark that this restriction for $\eta$ is temporarily needed for the proof--which is convenient for technical reasons--, whereas, in the proof of~\cite[Lemma 4.2]{CN26} we did not impose this one. Nevertheless, at the conclusion of the proofs of both theorems, $\eta$ must be chosen as small as desired.

We begin with the Caccioppoli inequality~\eqref{e.caccioppoli2} over two concentric cylinders $(1-\sigma)Q \subset Q$, where $Q:=B_\rho \times (0,\delta \rho^{sp}]$ and $\sigma \in (0,1)$ is to be selected below too and thereby obtain, for every $t \in (0,\delta \rho^{sp}]$, 
\begin{align}\label{e.mt1-est-NT-2}
\sfG_t&:=\int_{B_{(1-\sigma)\rho}\times \{t\}}\mathfrak{g}_-(u,k)\d{x}\notag\\[4pt]
&\quad\quad  \leq \int_{B_{(1-\sigma)\rho}\times \{0\}}\mathfrak{g}_-(u,k)\d{x}+C\frac{\rho^{(1-s)p}}{(\sigma \rho)^p}\iint_{Q}w_-^p\dxt \notag \\[4pt]
&\quad \quad \quad +C\frac{\rho^{d}}{(\sigma \rho)^{d+sp}}\iint_{Q}w_{-}(x,t)\d{x} \cdot \tail \big(w_-(t)\,; B_{\rho} \big)^{p-1}\d{t} \notag\\
&\quad \quad =:\sfG_0+\sfH+\sfI,
\end{align}
where the definitions of four integrations $\sfG_t, \sfG_0, \sfH$ and $\sfI$ are obvious from the context.  We proceed by estimating each of these four integrals separately. For $\sfG_t$, it is clear from the definition that
\begin{align*}
\sfG_t&=(p-1)\int_{B_{(1-\sigma)\rho}\times \{t\}} \int_u^k|\tau|^{p-2}(\tau-k)_-\dtaux\\
& \geq (p-1)\int_{B_{(1-\sigma)\rho}\times \{u(t) < \eta k\}} \int_{\eta k} ^k|\tau|^{p-2}(\tau-k)_-\dtaux.
\end{align*}
This, together with the simple estimate that
\[
\left|\{u(\cdot, t)<\eta k\} \cap B_\rho \right| \leq \left|\{u(\cdot, t)<\eta k\} \cap B_{(1-\sigma)\rho} \right|+d\sigma |B_\rho|,
\]
in turn implies that
\[
\sfG_t \geq (p-1) \Big(\left|\{u(\cdot, t)<\eta k\} \cap B_\rho \right|-d\sigma|B_\rho|\Big)\int_{\eta k} ^k|\tau|^{p-2}(\tau-k)_-\d{\tau}.
\]
For $\sfG_0$, in view of $u(\cdot, 0) \geq 0$ in $B_\rho \subset B_R$ and the assumption that
\[
\left|\{u(\cdot, 0)<k\} \cap B_\rho \right| \leq (1-\alpha)|B_\rho|
\]
we deduce that
\begin{align*}
\sfG_0 &=(p-1)\int_{B_{\rho}\times \{0\}} \int_u^k|\tau|^{p-2}(\tau-k)_-\dtaux\\
&\leq (p-1)\left|\{u(\cdot, 0)<k\} \cap B_\rho \right| \int_0^k|\tau|^{p-2}(\tau-k)_-\d{\tau}\\
&\leq (p-1)(1-\alpha)|B_\rho|\int_0^k|\tau|^{p-2}(\tau-k)_-\d{\tau}.
\end{align*}
By a exact same argument leading to the estimates of $\sfH+\sfI$, after enforcing~\eqref{e.mt1-est-NT-1}, we straightforwardly estimate that
\begin{align*}
\sfH+\sfI &\leq C \frac{\delta}{\sigma^p}k^p|B_\rho|+C\sigma^{-(d+sp)}\left(\delta+\delta^{\frac{\eps}{p-1+\eps}}\right)k^p|B_\rho|\\
&\leq C\frac{\delta^{\frac{\eps}{p-1+\eps}}}{\sigma^{d+p}}k^p|B_\rho|.
\end{align*}
Combining the previous three estimates with~\eqref{e.mt1-est-NT-2} and rearranging give
\begin{align}\label{e.mt1-est-NT-3}
\frac{\left|\{u(t)<\eta k\} \cap B_\rho \right| }{|B_\rho|} & \leq (1-\alpha) \left[1+\frac{\displaystyle \int_0^{\eta k}|\tau|^{p-2}(\tau-k)_-\d{\tau}}{\displaystyle \int_{\eta k}^k|\tau|^{p-2}(\tau-k)_-\d{\tau}}\right] \notag\\
& \quad \quad +C\frac{\delta^{\frac{\eps}{p-1+\eps}}}{\sigma^{d+p} \displaystyle \int_{\eta k}^k|\tau|^{p-2}(\tau-k)_-\d{\tau}}k^p+d\sigma.
\end{align}
It is straightforward to check that
\begin{equation}\label{e.mt1-est-NT-4} 
\int_0^{\eta k}|\tau|^{p-2}(\tau-k)_-\d{\tau} \leq k \int_0^{\eta k}|\tau|^{p-2}\d{\tau}=\frac{\eta}{p-1}k^p\eta^{p-1}.
\end{equation}
We next argue that
\begin{equation}\label{e.mt1-est-NT-5}
\int_{\eta k}^k|\tau|^{p-2}(\tau-k)_-\d{\tau} \geq ck^p
\end{equation}
for a constant $c(p)<\infty$. Indeed, by Lemma~\ref{t.g} we can estimate
\[
\int_{\eta k}^k|\tau|^{p-2}(\tau-k)_-\d{\tau}=\frac{1}{p-1}\mathfrak{g}_-(\eta k, k) \geq \frac{c_1}{p-1}\Big(|\eta k| +|k| \Big)^{p-2}\left(\eta k -k\right)_-^2
\]
hence we are led to consider two cases. In the first case $p \geq 2$, it is obvious that
\begin{align*}
\int_{\eta k}^k|\tau|^{p-2}(\tau-k)_-\d{\tau} &\geq \frac{c_1}{p-1} \left(\eta k-k\right)_-^p=\frac{c_1}{p-1}k^p(1-\eta)^p \\
&\geq \frac{c_1}{2^p(p-1)}k^p.
\end{align*}
In the remaining case $1<p<2$,  observe that, for $\eta \in (0,\nicefrac{1}{2})$,
\[
|\eta k| +|k| \leq (\eta +1)k \leq 3k/2 \quad \mbox{and} \quad (1-\eta)^2 \geq 1/4,
\]
thereby getting
\begin{align*}
\int_{\eta k}^k|\tau|^{p-2}(\tau-k)_-\d{\tau} \geq \frac{c_1}{p-1} \left(\frac{3k}{2}\right)^{p-2}k^2(1-\eta)^2 \geq \frac{c_1}{3^{2-p}2^p(p-1)}k^p.
\end{align*}
Thus, in all cases, we deduce~\eqref{e.mt1-est-NT-5}.

Combining the preceding estimates~\eqref{e.mt1-est-NT-3}--\eqref{e.mt1-est-NT-5} yields, for $c(p)<\infty$ and $C(\data)<\infty$, that
\[
\frac{\left|\{u(t)<\eta k\} \cap B_\rho \right| }{|B_\rho|} \leq (1-\alpha)(1+c\eta^p)+C\frac{\delta^{\frac{\eps}{p-1+\eps}}}{\sigma^{d+p}}+d\sigma.
\]
At this stage, we suppress $\eta<\nicefrac{1}{2}$ satisfying
\[
(1-\alpha)(1+c\eta^p) \leq 1-\frac{3}{4}\alpha \quad \iff \quad \eta \leq \left[\frac{\alpha}{4(1-\alpha)}\right]^{\nicefrac{1}{p}}
\]
and subsequently select
\[
\sigma=\frac{\alpha}{8d} \quad \mbox{and} \quad C\frac{\delta^{\frac{\eps}{p-1+\eps}}}{\sigma^{d+p}} \leq \frac{\alpha}{4},
\]
thereby implying the dependence of $\delta$ on $\alpha$; namely that $\delta=\delta_0\alpha^{(d+p+1)\frac{p-1+\eps}{\eps}}$ for a positive constant $\delta_0$ depending upon $\data$ and $\eps$, and moreover we have that
\[
\frac{\left|\{u(t)<\eta k\} \cap B_\rho \right| }{|B_\rho|} \leq 1-\frac{\alpha}{2}, \quad \forall t \in (0,\delta \rho^{sp}].
\]
This finishes the proof.
\end{proof}

We are ready to derive the following measure shrinking lemma.

\begin{lemma}\label{Lm:shrinking-FT}
Let $u$ be a weak super-solution to~\eqref{e.NT}-\eqref{e.kernel}, in the sense of Definition~\ref{def-of-NT}, satisfying $u \geq 0$ in $\mathcal{Q}_R$. Let $\eps \in (0,\infty)$, $k>0$ and $\alpha \in (0,1]$. Suppose that, for some $\delta \in (0,1)$ and $\sigma \in (0,\nicefrac{1}{2})$,
\begin{equation}\label{e.shrinking-est-0-NT}
\Big|\left\{u(\cdot,t) \geq k \right\} \cap B_\rho(x_0)\Big| \geq \alpha |B_\rho| \quad \mbox{for all\,\, $t \in (t_0-\delta\rho^{sp},t_0]$}.
\end{equation}
Then, there exists $C>0$ depending only on $\data$ and independent of $\alpha$ and $\delta$, such that if for some $\sigma \in (0, 1/2)$
\[
\left(\frac{\rho}{R}\right)^{\frac{sp\eps}{(p-1)(p-1+\eps)}} \left(\dashint_{I_R(t_0)} \tail \big(u_-(t)\,;B_R(x_0)\big)^{p-1+\eps}\d{t}\right)^{\nicefrac{1}{(p-1+\eps)}} \leq \sigma k
\]
is satisfied, then
\[
\Big|\left\{ u \leq \sigma k \right\} \cap Q_{\rho,\tau}(z_0) \Big| \leq C\frac{\sigma^{p-1}}{\delta \alpha}|Q_{\rho,\tau}(z_0)|
\]
holds for every $Q_{2\rho,2\tau}(z_0) \subset \cQ_R \Subset \Omega_T$ with $\tau=\delta \rho^{sp}$.
\end{lemma}
\begin{proof}
As usual, to lighten the notation, we drop the center $z_0$ from cylinders considered here.
Set
\[
w_{\pm}:=\left(u-\sigma k\right)_\pm \quad \mbox{and} \quad Q_{\rho,\tau}:=B_\rho \times (-\tau, 0] \quad \mbox{with \,\,$\tau:=\delta\rho^{sp}$}. 
\]
We apply~\eqref{e.caccioppoli2}$_-$ on the concentric cylinders $Q_{\rho, \tau} \subset Q_{2\rho,\tau}$ to obtain that
\begin{align*}
&\int_{Q_{\rho,\tau}}w_-(x,t)\d{x}\left(\int_{\R^d}\frac{w_+^{p-1}(y,t)}{|x-y|^{d+sp}}\d{y}\right)\d{t}\\
&\quad \leq \int_{B_{2\rho} \times \{0\}}\mathfrak{g}_-(u,\sigma k)\d{x}+c\rho^{-sp}\iint_{Q_{2\rho,\tau}}w_-^p\dxt \\
&\quad \quad \quad \quad +c\rho^{-sp} \iint_{Q_{2\rho,\tau}}w_-\d{x}\cdot \tail\big(w_-(t)\,;B_{2\rho}\big)^{p-1}\d{t}.
\end{align*}
A completely same argument in the proof of~\cite[Lemma 4.3]{CN26} reduces our task here to estimating the first term on the right side. From Lemma~\ref{t.g} it follows that
\begin{align*}
\sfG_0:=\int_{B_{2\rho} \times \{0\}}\mathfrak{g}_-(u,\sigma k)\d{x} \leq C(p) \int_{B_{2\rho} \times \{0\}} \Big(|u|+|\sigma k|\Big)^{p-2}(u-\sigma k)_-^2\d{x}.
\end{align*}
Similarly to $\sfG$ as in the proof of Lemma~\ref{t.DGtype-NT}, the proof splits into two cases: $p \geq 2$ and $1<p<2$. In the case $p \geq 2$, since by the nonnegativity of $u$ on $B_R$, $0 \leq u \leq \sigma k $ on $B_{2\rho} \cap \{u(\cdot, 0) <\sigma k\}$ we have that
\[
\sfG_0\leq C(\sigma k)^{p}|B_{2\rho}| = C\frac{(\sigma k)^p}{\delta \rho^{sp}}|Q_{\rho,\tau}|.
\]
In the opposite case $1<p<2$, obviously, 
\[
\sfG_0 \leq C\int_{B_{2\rho} \times \{0\}} (u-\sigma k)_-^p\d{x} \leq C\frac{(\sigma k)^p}{\delta \rho^{sp}}|Q_{\rho,\tau}|
\]
follows. In all cases, we have shown that
\[
\sfG_0=\int_{B_{2\rho} \times \{0\}}\mathfrak{g}_-(u,\sigma k)\d{x} \leq C\frac{(\sigma k)^p}{\delta \rho^{sp}}|Q_{\rho,\tau}|.
\]
The remaining step is exactly the same as the one of~\cite[Lemma 4.3]{CN26}, which concludes the proof.
\end{proof}



Using the advantage that the integral over whole $\R^d$ on the left side of~\eqref{e.caccioppoli2}$_-$, another shriking property, peculiar of the nonlocal framework, can be obtained.

\begin{lemma}\label{t.newmt-NT}
Fix the cylinder $Q_\rho \equiv Q_\rho(z_0) \subset \cQ_R$ with $z_0=(x_0,t_0)\in \Omega_T$. Let $u$ be a weak super-solution to~\eqref{e.NT}-\eqref{e.kernel}, in the sense of  Definition~\ref{def-of-NT}, satisfying $u \geq 0$ in $\mathcal{Q}_R$. Fix $\eps \in (0,\infty)$ and $k>0$ arbitrarily. Suppose that
\[
\left(\frac{\rho}{R}\right)^{\frac{sp\eps}{(p-1)(p-1+\eps)}} \left(\dashint_{I_R(t_0)} \tail \big(u_-(t)\,;B_R(x_0)\big)^{p-1+\eps}\d{t}\right)^{\nicefrac{1}{(p-1+\eps)}} \leq k.
\]
There exists a constant $C(\data)>0$ such that 
\[
\inf_{t \in (t_0-\rho^{sp},t_0]}\Big|\left\{u(\cdot,t) \leq k \right\} \cap B_\rho(x_0) \Big| \leq  \left[\frac{C k^{p-1}}{(u^{p-1})_{Q_\rho(z_0)}}\wedge \frac{C k^{p-1}}{\displaystyle \dashint_{t_0-\rho^{sp}}^{t_0} \tail \big(u_+(t)\,;B_\rho(x_0)\big)^{p-1}\d{t}} \right]|B_\rho(x_0)|.
\]
\end{lemma}
\begin{proof}
The proof is almost repeated verbatim for~\cite[Lemmas 5.1 and 5.2]{CN26}; whereas the different part is that
\[
\int_{B_\rho \times \{0\}} \mathfrak{g}_-(u,k)\d{x} \leq Ck^{p}|B_{\rho}|.
\]
This is proved in the same way as $\sfG_0$ as in the previous lemma, completing the proof.
\end{proof}

We conclude this subsection by exhibiting a De Giorgi type lemma that requires a weaker assumption, when compared to Lemma~\ref{t.DGtype-NT}. 

\begin{lemma}\label{t.weakDGtype-NT}
Let $u$ be a weak super-solution to~\eqref{e.NT}-\eqref{e.kernel}, in the sense of  Definition~\ref{def-of-NT}, satisfying $u \geq 0$ in $\mathcal{Q}_R$ and let $\eps \in (0,\infty)$. There exist parameters $\delta$, $\nu_\ast \in (0,1)$, both depending only on $\data$ and $\eps$, such that, if
\[
\left|\{u(\cdot,t_0)<k\} \cap B_{2\rho}(x_0)\right| \leq \nu_\ast \left|B_{2\rho}(x_0)\right|
\]
and
\[
\left(\frac{\rho}{R}\right)^{\frac{sp\eps}{(p-1)(p-1+\eps)}} \left(\dashint_{I_R(t_0)} \tail \big(u_-(t)\,; B_R(x_0)\big)^{p-1+\eps}\d{t}\right)^{\nicefrac{1}{(p-1+\eps)}} \leq \frac{k}{4}
\]
then 
\[
u \geq \frac{k}{4} \quad \textrm{a.e.\,\,in}\,\, \,B_{\frac{\rho}{2}}(x_0) \times \left(t_0+\frac{1}{2}\delta \rho^{sp}, t_0+\delta\rho^{sp}\right],
\]
provided that $B_{2\rho}(x_0) \times \left(t_0,t_0+\delta\rho^{sp}\right] \subset \cQ_R \Subset \Omega_T$. 
\end{lemma}
\begin{proof}
The arguments are identical to those in~\cite[Lemmas 5.1 and 5.2]{CN26}, modulo the modifications required for the $\mathfrak{g}_-$ term as detailed in Lemma~\ref{t.DGtype-NT}.
\end{proof}

\subsection{Nonlocal weak Harnack inequality}

Having the above measure theoretical lemmas at hand, we deduce the subsequent nonlocal weak Harnack inequality.



\begin{theorem}[Weak Harnack inequality]\label{t.weakHarnack2-NT}
Let $u $ be a weak super-solution to~\eqref{e.NT}-\eqref{e.kernel} in the sense of Definition~\ref{def-of-NT} such that $u \geq 0$ in $\cQ_R$. Fix $\eps \in (0,\infty]$. There exists $\eta \in (0,1)$ depending only on $\data$ and $\eps$ such that
\begin{align*}
\inf_{B_\rho(x_0)} u(t)&+\left(\frac{\rho}{R}\right)^{\frac{sp\eps}{(p-1)(p-1+\eps)}}\left(\dashint_{I_R(t_0)}\tail \big(u_-(t)\,;B_R(x_0)\big)^{p-1+\eps}\d{t}\right)^{\nicefrac{1}{(p-1+\eps)}} \\
& \geq \eta \left[\left(\biint_{Q_{2\rho}(z_0)}u^{p-1}\dxt \right)^{\nicefrac{1}{(p-1)}}+\left(\dashint_{t_0-(2\rho)^{sp}}^{t_o} \tail \big(u_+(t)\,;B_{2\rho}(x_0)\big)^{p-1}\d{t} \right)^{\nicefrac{1}{(p-1)}}\right]
\end{align*}
for almost every $t  \in \left(t_0+\frac{3}{4}(4\rho)^{sp}, t_0+(4\rho)^{sp} \right]$, provided that
\[
B_{2\rho} (x_0)\times \left(t_0-(2\rho)^{sp}, t_0+6(4\rho)^{sp}\right] \subset \cQ_R \Subset \Omega_T.
\]
\end{theorem}

The proof is a combination of Lemmas~\ref{t.newmt-NT},~\ref{t.weakDGtype-NT} and Theorem~\ref{t.boundedness-NT}. The reader can also consult the argument~\cite[Section 5.2]{CN26} for a more precise proof.

\subsection{Proof of Theorem~\ref{t.fullHarnack-NT}}

\begin{proof}[Proof of Theorem~\ref{t.fullHarnack-NT}]
The proof is in the style of Trudinger's approach, since it uses the chaining of the Weak Harnack inequality with a refined sup-bound. Nevertheless, the proved strong measure-theoretical properties allow for a Weak Harnack inequality with a positive exponent and the control of the tail of $u_+$. Let's start. Without loss of generality, we may assume $z_0=(x_0,t_0)=(0,0)$. Let \[M:=\sup_{Q_{\rho, \nicefrac{(2\rho)^{sp}}{2}} }u<\infty\,.\] We assume that $u$ is nonnegative $\cQ_{R}$ with $R \in (0, \infty)$ being specified later. Appealing to Theorem~\ref{t.boundedness-NT} with $\nu=p-1$, $\sigma=\nicefrac{1}{2}$ , $\theta=(2\rho)^{sp}$ and replacing $\rho$ by $2\rho$. Notice that the relevant constants depend only on $\data$. Thus, we have that
\begin{align*}
M \leq C\left(\dashint_{-(2\rho)^{sp}}^0 \tail\big(u_+(t) ; B_\rho\big)^{p-1}\d{t}\right)^{\nicefrac{1}{(p-1)}}+C\left(\biint_{Q_{2\rho}}u_+^{p-1}\dxt\right)^{\nicefrac{1}{(p-1)}}
\end{align*}
for a constant $C(\data)<\infty$. As for the tail term, we can estimate that
\begin{align*}
\tail\big(u_+(t)\,; B_\rho\big)^{p-1}
&\leq \frac{1}{2^{sp}}\tail\big(u_+(t) ; B_{2\rho}\big)^{p-1}+C(d)\dashint_{B_{2\rho}}u_+^{p-1}(t)\d{x}.
\end{align*}
On the other hand, by Theorem~\ref{t.weakHarnack2-NT} there exists $\eta(\data) \in (0,1)$ such that
\begin{align*}
&\inf_{B_\rho\times \left(\frac{3}{4}(4\rho)^{sp}, (4\rho)^{sp} \right]}u+\left(\frac{\rho}{R}\right)^{\frac{sp\eps}{(p-1)(p-1+\eps)}}\left(\dashint_{I_R} \tail \big(u_-(t)\,;B_R\big)^{p-1+\eps}\d{t}\right)^{\nicefrac{1}{(p-1+\eps)}} \\
& \quad \quad \quad \geq \eta \left[\left(\biint_{Q_{2\rho}}u^{p-1}\dxt \right)^{\nicefrac{1}{(p-1)}}+\left(\dashint_{-(2\rho)^{sp}}^{0} \tail \big(u_+(t)\,;B_{2\rho}\big)^{p-1}\d{t} \right)^{\nicefrac{1}{(p-1)}}\right],
\end{align*}
provided that $B_{2\rho}\times \left(-(2\rho)^{sp}, 6(4\rho)^{sp}\right] \subset \cQ_R \Subset \Omega_T$. Now, we control the tail as follows:
\begin{align}\label{e.tail.result}
&\left(\dashint_{-\frac{1}{2}(2\rho)^{sp}}^0 \tail \big(u_+(t) ; B_\rho \big)^{p-1}\d{t}\right)^{\nicefrac{1}{(p-1)}} \notag\\
&\quad \quad \quad \leq C(\data)\left[M+\left(\frac{R}{\rho}\right)^{\nicefrac{sp}{(p-1+\eps)}}\left(\dashint_{I_R} \tail \big(u_-(t)\,;B_R\big)^{p-1+\eps}\d{t}\right)^{\nicefrac{1}{(p-1+\eps)}}\right]. 
\end{align}
In fact, we apply~\eqref{e.caccioppoli2}$_-$ with $Q_{\rho,\frac{1}{2}(2\rho)^{sp}} \subset Q_{2\rho,(2\rho)^{sp}}$ and $k=2M$. Using this and the fact that $u \geq 0$ in $Q_{2\rho,(2\rho)^{sp}} \subset \cQ_R$, we estimate that, for a constant $c(\data)<\infty$, 
\begin{align*}
&\iint_{Q_{\rho,\frac{1}{2}(2\rho)^{sp}}}(u(x,t)-2M)_-\left[\int_{\R^d} \frac{(u(y,t)-2M)_+^{p-1}}{|x-y|^{d+sp}}\d{y}\right]\dxt\\
&\quad  \leq \frac{c}{\rho^{sp}}\iint_{Q_{2\rho,(2\rho)^{sp}}}(u-2M)_-^p\dxt \\
&\quad \quad \quad \quad +\frac{c}{\rho^{sp}}\iint_{Q_{2\rho,(2\rho)^{sp}}}(u-2M)_-\tail \big((u-2M)_-\,;B_{2\rho}\big)^{p-1}\dxt\\
&\quad  \leq  cM^p\rho^d+cM\rho^d \dashint_{-(2\rho)^{sp}}^0\tail \big(u_-(t)\,;B_{2\rho}\big)^{p-1}\d{t}.
\end{align*}
We next estimate the left side of the last inequality. Observe that
\[
|x-y| \leq 2|y|, \quad \forall x \in B_\rho, \quad \forall y \in \R^d \setminus B_\rho.
\]
This, together with the fact that $0 \leq u \leq M$ in $Q_{\rho,\frac{1}{2}(2\rho)^{sp}} \subset \cQ_R$ and the following elemental inequality~\cite[Lemma 4.4]{Coz17}
\[
(u(y,t)-2M)_+^{p-1} \geq (1 \wedge  2^{2-p}) u_+(y,t)^{p-1}-(2M)^{p-1},
\]
yields that
\begin{align*}
\iint_{Q_{\rho,\frac{1}{2}(2\rho)^{sp}}}&(u(x,t)-2M)_-\left[\int_{\R^d} \frac{(u(y,t)-2M)_+^{p-1}}{|x-y|^{d+sp}}\d{y}\right]\dxt \\
&\geq cM \left[|B_\rho|\int_{-\frac{1}{2}(2\rho)^{sp}}^0 \left(\int_{\R^d \setminus B_\rho} \frac{u_+(y,t)^{p-1}-(2M)^{p-1}}{(2|y|)^{d+sp}}\d{y}\right)\d{t}\right] \\
&\geq cM\rho^d \dashint_{-\frac{1}{2}(2\rho)^{sp}}^0 \tail \big(u_+(t)\,;B_{\rho}\big)^{p-1}\d{t}-cM^p\rho^d
\end{align*}
for a constant $c(d,s,p)<\infty$. After combining the above displays, rearranging provides that
\begin{align*}
&\left(\dashint_{-\frac{1}{2}(2\rho)^{sp}}^0\tail \big(u_+(t)\,;B_{\rho}\big)^{p-1}\d{t}\right)^{\nicefrac{1}{(p-1)}} \\
& \quad \quad \quad \quad \leq c \left[M+\left(\dashint_{-(2\rho)^{sp}}^0\tail \big(u_-(t)\,;B_{2\rho}\big)^{p-1}\d{t} \right)^{\nicefrac{1}{(p-1)}}\right].
\end{align*}
Finally, rearranging the tail term on the right side and using H\"{o}lder inequality conclude~\eqref{e.tail.result}.

Altogether, combining the preceding three estimates and rearranging yield that
\begin{align*}
&\sup_{B_{\rho} \times \left(-\frac{1}{2}(2\rho)^{sp}, 0\right]}u +\left(\dashint_{-\frac{1}{2}(2\rho)^{sp}}^{0} \tail \big(u_+(t) ; B_\rho \big)^{p-1}\d{t}\right)^{\frac{1}{p-1}} \notag\\[2pt]
& \quad \leq C_\mathsf{H}\left[ \inf_{B_\rho \times \left(\frac{3}{4}(4\rho)^{sp}, (4\rho)^{sp} \right]}u \right. \notag\\[2pt]
&\left. \quad \quad \quad \quad \quad +\left[\left(\frac{R}{\rho}\right)^{\frac{sp}{p-1+\eps}}+\left(\frac{\rho}{R}\right)^{\frac{sp\eps}{(p-1)(p-1+\eps)}} \right]\left(\dashint_{I_R} \tail \big(u_-(t)\,;B_R\big)^{p-1+\eps}\d{t}\right)^{\frac{1}{p-1+\eps}}\right]\,.
\end{align*}
Finally, selecting $R=4\cdot 6^{\nicefrac{1}{sp}}\rho=:R_0$ in the last display concludes, taking into account that $B_{2\rho}\times \left(-(2\rho)^{sp}, 6(4\rho)^{sp}\right] \subset \cQ_{R_0}$, that the statement of Theorem~\ref{t.fullHarnack-NT}. The proof is complete.
\end{proof}
\makeatletter
\renewcommand{\thesection}{\Alph{section}.\arabic{subsection}}
\makeatother

\appendix

\makeatletter
\renewcommand{\thefigure}{\Alph{section}.\arabic{figure}}
\@addtoreset{figure}{section}
\makeatother

\subsubsection*{\bf Acknowledgments}
The authors appreciate the \emph{Erwin Schr\"{o}dinger Institut f\"{u}r Mathematik und Physik of Vienna} for its kind hospitality during the Workshop ``Degenerate
and Singular PDEs'' held in 24--28, February 2025, where this collaboration began. S.C. acknowledges the partial funding of GNAMPA (INdAM), and the support department of Mathematics of the university of Bologna.



\end{document}